\documentclass[11pt]{amsart} 

\usepackage{amsmath, amssymb, amsthm}

\usepackage{geometry}
\usepackage{enumitem} 
\usepackage{braket}
\usepackage{eucal}
\usepackage[mathletters]{ucs}
\usepackage[utf8x]{inputenc}

\usepackage[normalem]{ulem}

\usepackage{tikz}
\usepackage{tikz-cd}

\DeclareMathOperator{\sSet}{\mathcal{S}et_{\Delta}} 
\DeclareMathOperator{\Cat}{\mathcal{C}at} 
\DeclareMathOperator{\SCat}{\mathcal{C}at_{\Delta}}
\DeclareMathOperator{\calC}{\mathcal{C}}
\DeclareMathOperator{\calA}{\mathcal{A}}  
\DeclareMathOperator{\calB}{\mathcal{B}}
\DeclareMathOperator{\calD}{\mathcal{D}}

\DeclareMathOperator{\calF}{\mathcal{F}}
\DeclareMathOperator{\calH}{\mathcal{H}}
\DeclareMathOperator{\calJ}{\mathcal{J}}
\DeclareMathOperator{\calW}{\mathcal{W}}
\DeclareMathOperator{\sk}{\mathrm{sk}}
\DeclareMathOperator{\cosk}{\mathrm{cosk}}
\DeclareMathOperator{\Map}{\mathrm{Map}}
\DeclareMathOperator{\Fun}{\mathrm{Fun}}
\DeclareMathOperator{\Hom}{\mathrm{Hom}}

\DeclareMathOperator{\h}{\mathrm{h}} 
\DeclareMathOperator{\id}{id}

\theoremstyle{plain}
\newtheorem{thmx}{Theorem}

\newtheorem{theorem}{Theorem}[section]
\newtheorem{lemma}[theorem]{Lemma}
\newtheorem{proposition}[theorem]{Proposition}
\newtheorem{corollary}[theorem]{Corollary}

\newtheorem*{proposition*}{Proposition}

\theoremstyle{definition}
\newtheorem{definition}[theorem]{Definition}
\newtheorem{remark}[theorem]{Remark} 
\newtheorem{example}[theorem]{Example} 
\newtheorem{notation}[theorem]{Notation}
\newtheorem{construction}[theorem]{Construction}

\begin{document} 

\title[Notes on the Joyal model structure]{Notes on the Joyal model structure}
  \author[D.\ Stevenson]{Danny Stevenson}
  \address[Danny Stevenson]
  {School of Mathematical Sciences\\
  University of Adelaide\\
  Adelaide, SA 5005 \\
  Australia}
  \email{daniel.stevenson@adelaide.edu.au}

\thanks{This research was supported under the Australian
Research Council's {\sl Discovery Projects} funding scheme (project number DP180100383).}

\subjclass[2010]{55U10, 55U35, 18G30}

\begin{abstract}
We give a new construction of the 
Joyal model structure on the category $\mathcal{S}\mathrm{et}_{\Delta}$ 
of simplicial sets, and we provide a simple characterization of the fibrations in  
it.  We characterize the inner anodyne maps 
between simplicial sets in terms of categorical equivalences and use this characterization to establish 
the inner model structure on the category $\mathcal{S}\mathrm{et}_{\Delta}(O)$ 
of simplicial sets whose set of zero-simplices is equal to a fixed set $O$.       
\end{abstract}
\maketitle

\section{Introduction} 
\label{sec:intro}

Recall that a simplicial set 
$S$ is said to be a {\em quasi-category} 
if every map $\Lambda^n_i\to S$ 
with $0<i<n$ can be extended  
along the inclusion $\Lambda^n_i\subseteq 
\Delta^n$ to give an $n$-simplex 
of $S$.  Quasi-categories, or restricted  
Kan complexes,  
were introduced by 
Boardman and Vogt in their 
work on homotopy invariant 
algebraic structures \cite{BV}.    

It is a remarkable fact, and a 
great insight of Joyal, that most 
concepts and results from ordinary category 
theory can be extended to quasi-categories, 
see \cite{J2}.  Lurie has vastly extended 
Joyal's theory in the books \cite{HTT} and \cite{HA}.    
We will follow Lurie and use 
the term $\infty$-{\em category} instead of 
quasi-category.  

One of the key discoveries of Joyal 
was the existence of a primordial model 
structure on the category $\sSet$ of simplicial 
sets in which the $\infty$-categories 
are the fibrant objects.  This model 
structure is called the {\em Joyal model 
structure} and it plays an important role in the 
Joyal/Lurie theory of $\infty$-categories.  
For example, it makes precise the idea that 
every simplicial set $S$ presents 
an $\infty$-category by a generators 
and relations type construction.  A choice of a 
fibrant replacement for $S$ in the Joyal  
model structure gives an $\infty$-category 
$\calC$, well defined up to equivalence, 
which one may think of as an $\infty$-category 
generated by $S$.  

While the Joyal model 
structure on $\sSet$ is known to be cofibrantly 
generated, it is an open problem to describe an explicit set of generating 
acyclic cofibrations (compare with Remark 2.14 in \cite{DS}).  
Also, while the fibrations between fibrant 
objects in the Joyal 
model structure are well understood 
(see Corollary 2.4.6.5 of \cite{HTT}), 
an explicit description of the fibrations 
in general
is unknown.  This is due to the fact 
that the existing treatments of the Joyal 
model structure (see \cite{J2,HTT} and 
also the streamlined treatment in \cite{DS}) 
rely on set theoretic arguments 
which place a greater emphasis on the 
role of weak equivalences (compare with 
Proposition A.2.6.13 of \cite{HTT}).     

In this paper we will attempt to address  
these questions with the 
following theorems.  

\begin{thmx} 
\label{thm:A}
A map $p\colon X\to S$ in 
$\sSet$ is a fibration in the Joyal model structure  
if and only if it is an inner fibration 
and the induced functor $\h(p)\colon 
\h(X)\to \h(S)$ on homotopy categories is an isofibration.  
\end{thmx} 

Here $\h\colon \sSet\to \Cat$ denotes 
the functor which sends a simplicial 
set to its homotopy category (see Section 
1.2.3 of \cite{HTT}).  

\begin{thmx} 
\label{thm:B}
A set of generating 
trivial cofibrations for the Joyal model 
structure on $\sSet$ is given by 
the following collection of morphisms: 
\begin{itemize} 
\item the inner horn inclusions 
$\Lambda^n_i\subseteq \Delta^n$, $0<i<n$ for $n\geq 2$; 
\item a set of representatives 
for the isomorphism classes of 
simplicial sets $B$ with two vertices $0$ and $1$, 
with countably many non-degenerate simplices and 
such that the inclusion $\set{0}\hookrightarrow B$ 
is a categorical equivalence.  
\end{itemize}
\end{thmx} 

In order to prove these theorems 
we give a characterization of the 
inner anodyne maps in $\sSet$ which 
is of interest in its own right. Recall 
that the class of inner anodyne maps 
in $\sSet$ is the saturated class 
generated by the inner horn inclusions 
$\Lambda^n_i \subseteq \Delta^n$ for 
$0<i<n$. 

For instance, if $S$ is a simplicial set, 
then we can apply Quillen's small object 
argument to the set of these inner 
horn inclusions to produce an 
inner anodyne map $S\to \calC$, where 
$\calC$ is an $\infty$-category.  We may 
think of this $\infty$-category as 
being `freely generated' by $S$.  

In Section~\ref{sec:inner anodyne stuff} below we will prove 
the following theorem which answers 
a question left open by Joyal (see paragraph 2.10 of 
\cite{J3}).  

\begin{thmx}
\label{thm:C} 
A monomorphism in $\sSet$ is 
inner anodyne if and only if 
it is bijective on vertices 
and it is a categorical equivalence.  
\end{thmx} 

Note that this theorem has the 
following interesting consequence 
(see Proposition~\ref{prop:equiv statements for ia}): the class 
of inner anodyne maps in $\sSet$ 
satisfies the 2-out-of-3 property 
in the class of monomorphisms 
in $\sSet$.  In other words, 
if $u\colon A\to B$ and 
$v\colon B\to C$ are 
monomorphisms in $\sSet$ such 
that any two of the maps $u$, 
$v$ and $vu$ are inner 
anodyne, then all three maps 
are inner anodyne.  
It is clear that the class of 
inner anodyne maps is closed 
under composition.  The fact 
that the class of inner anodyne 
maps satisfies the so-called 
right cancellation property 
was established in \cite{S2}.  
It is not obvious that the 
analogous left cancellation property 
is also satisfied.  

As an application of Theorem~\ref{thm:C}, 
we construct what we call the 
{\em inner model structure} on the 
category $\sSet(O)$, consisting of 
all simplicial sets whose set of 
$0$-simplices is equal to a fixed set $O$ 
and whose morphisms are the 
simplicial maps which induce 
the identity on $O$.  

More precisely, we prove 
the following result in 
Section~\ref{sec:inner anodyne stuff}. 

\begin{thmx} 
\label{thm:D}
There is the structure of a 
left proper, cofibrantly generated 
model category on $\sSet(O)$ for 
which
\begin{itemize} 
\item the weak equivalences are the 
categorical equivalences in $\sSet(O)$; 
\item the cofibrations are the monomorphisms 
in $\sSet(O)$; and 
\item the fibrant objects are the 
$\infty$-categories whose 
set of objects is equal to $O$.  
\end{itemize}
\end{thmx} 

It turns out that the fibrations 
in this model structure are the inner 
fibrations in $\sSet(O)$.  The existence 
of this model structure explains 
the fact that inner anodyne maps 
satisfy the 2-out-of-3 property.  
We make use 
of it in our proof of 
Theorem~\ref{thm:B}.

To put this theorem into context,  
recall that in the paper \cite{DK}, Dwyer and 
Kan constructed a model structure 
on the category $\SCat(O)$ of 
simplicial categories with a 
fixed object set $O$.  
This model structure plays 
a role in the construction 
of the Bergner model 
structure on $\SCat$.  





We summarize the contents of this 
paper.  In Section~\ref{sec:categorical equivalences} 
we review some basic properties of categorical 
equivalences --- the weak equivalences in the 
Joyal model structure.  Everything in this section 
is well-known and due in its original form to 
Joyal.  We have included proofs of some of 
the statements here in an effort to make 
the construction of the Joyal model structure 
as self contained as possible (see 
Remark~\ref{rem:ingredients}).  In 
Section~\ref{sec:pre-fibrant simplicial sets} 
we introduce the concept of what we have 
called (for lack of a better name) {\em pre-fibrant} 
simplicial sets.  This is an ad-hoc device 
which we have introduced to counter the 
fact that the naive mapping spaces of a simplicial 
set $S$ do not in general have the homotopy type of the 
mapping spaces of the $\infty$-category 
generated by $S$ (see 
Remark~\ref{rem:naive mapping spaces}).  The key result that 
we prove here is Proposition~\ref{prop:key prop}, 
which shows that every pre-fibrant simplicial 
set $S$ freely generates an $\infty$-category 
whose mapping spaces are isomorphic to 
the naive mapping spaces of $S$.   

In Section~\ref{sec:descent} we study 
a notion of descent for inner fibrations.  
Our main result is 
Proposition~\ref{prop:descent pre-fibrant simp sets}.  
This states that given an inner fibration 
$p\colon X\to S$, there exists an 
inner anodyne map $S\to T$, where $T$ 
is a pre-fibrant simplicial set, and an inner 
fibration $q\colon Y\to T$ forming part 
of a pullback diagram 
\[
\begin{tikzcd} 
X \arrow[d,"p"] \arrow[r] & Y \arrow[d,"q"] \\ 
S \arrow[r] & T 
\end{tikzcd} 
\]
This result is useful when one wants to 
prove that the inner fibration $p\colon X\to S$ 
is a left fibration or a trivial fibration, as 
one can then translate this problem into the 
problem of proving that the inner fibration 
$q\colon Y\to T$ is a left fibration or a trivial 
fibration.  This problem is sometimes easier 
to solve since $T$ (and hence $Y$) are pre-fibrant 
simplicial sets.  In 
Section~\ref{sec:inner anodyne stuff} 
we apply these results to prove 
Theorem~\ref{thm:C} and 
establish the inner model structure 
(Theorem~\ref{thm:D}).  
In Section~\ref{sec:joyal model structure} 
we prove Theorem~\ref{thm:A} 
and Theorem~\ref{thm:B}, 
and establish the 
Joyal model structure (Theorem~\ref{thm:joyal nodel str}).  
Finally, in Sections~\ref{sec:technical lemmas} --~\ref{sec:proof of bergners lemma} 
we prove some technical results 
that are needed earlier in the paper.

\medskip 

\noindent
{\bf Notation}: for the most part we will 
adopt the notation and terminology from 
\cite{HTT}.  Thus $\sSet$ will denote the 
category of simplicial sets for instance.  
Typically $\infty$-categories will  
be denoted with calligraphic letters, $\mathcal{C}$, 
$\mathcal{D}$, \ldots etc, while simplicial 
sets will typically be denoted with capital 
Latin letters, $S$, $T$, \ldots etc.  We 
follow the standard convention of not 
distinguishing notationally between an ordinary 
category and its nerve.

\section{Categorical equivalences} 
\label{sec:categorical equivalences} 

In this section we recall the notion 
of categorical equivalences between simplicial 
sets and give some examples.  We prove 
that a map between $\infty$-categories  
is a categorical equivalence if and only 
if it is a Dwyer-Kan equivalence (Definition~\ref{def:DK equivalence}).  
Everything in this section is 
well-known and is due to Joyal and 
Lurie.  

\subsection{The homotopy category of a simplicial set} 
\label{subsec:hty cats} 
Recall (see Section 1.2.3 of \cite{HTT}) the 
functor $\h\colon \sSet\to \Cat$ which sends 
a simplicial set $S$ to its {\em homotopy category} 
$\h(S)$.  Recall also that an edge $e\colon x\to y$ 
in an $\infty$-category $\calC$ is said to be 
an {\em equivalence} if its image in the associated 
homotopy category $\h(\calC)$ is an isomorphism.  
More generally we will say that an edge 
$e\colon x\to y$ in a simplicial set $S$ is 
an equivalence if its image in $\h(S)$ is an 
isomorphism.  

An edge $e\colon x\to y$ in an $\infty$-category 
$\calC$ is an equivalence if and only if the 
$1$-simplex $e\colon \Delta^1\to \calC$ classifying 
$e$ extends along the canonical map $\Delta^1\to J$, 
where $J$ denotes the {\em groupoid interval} 
(i.e.\ the groupoid with two objects $0$ and $1$ 
and a unique isomorphism between them).
Clearly 
the existence of 
such an extension is a sufficient 
condition; to see that it 
is also necessary observe 
that if $e\colon \Delta^1\to 
\calC$ is an equivalence then 
$e$ factors through $\calC^{\simeq}$, 
the maximal Kan complex contained 
in $\calC$.  Since 
$\Delta^1\to J$ is anodyne 
it follows that the map 
$\Delta^1\to \calC^{\simeq}$ 
extends along the 
inclusion $\Delta^1\to J$.

\subsection{Equivalences of $\infty$-categories}
We recall the notion of 
an equivalence between $\infty$-categories 
and give some examples.   

\begin{definition} 
\label{def:equiv of oo-cats}
Let $\calC$ and $\calD$ be $\infty$-categories.  A map 
$f\colon \calC\to \calD$ is said to be an {\em equivalence} 
if there exists a map $g\colon \calD\to \calC$ such that 
there are equivalences $gf\to \id_{\calC}$ and $fg\to \id_{\calD}$ 
in the $\infty$-categories $\Fun(\calC,\calC)$ 
and $\Fun(\calD,\calD)$ respectively.  
\end{definition} 

Following Dugger and Spivak \cite{DS} 
we introduce the following notation.  

\begin{notation} 
We say that maps $f,g\colon S\to T$ in $\sSet$ 
are $J$-{\em homotopic} if there exists a 
map $H\colon S\times J\to T$ such that 
$Hi_0 = f$ and $Hi_1 = g$, where $i_0,i_1\colon 
S\to S\times J$ are the obvious inclusions.  
\end{notation}

\begin{remark} 
Thus $f\colon \calC\to \calD$ is an 
equivalence of $\infty$-categories if and only 
if it is a {\em $J$-homotopy equivalence} 
in the sense that there is a map $g\colon \calD\to \calC$ 
such that the composites $fg$ and $gf$ are 
$J$-homotopic to the respective identity maps.  
\end{remark}

\begin{lemma} 
\label{lem:hty equiv between Kan complexes}
Let $f\colon K\to L$ be a homotopy equivalence, 
where $K$ and $L$ are Kan complexes.  Then 
$f$ is an equivalence of $\infty$-categories.  
\end{lemma} 

\begin{proof} 
This follows immediately 
from the fact that every edge 
in a Kan complex is an 
equivalence.  
\end{proof}

\begin{example} 
\label{ex:0 <J cat equiv}
The inclusion $\set{0}\subseteq J$ 
is an equivalence. 
\end{example}

\subsection{Categorical equivalences}
We recall the definition and some properties of 
the class of categorical equivalences in $\sSet$.  

\begin{definition}[Joyal] 
A map $f\colon S\to T$ of simplicial sets is 
said to be a {\em categorical equivalence} 
if for any $\infty$-category $\calC$, the induced 
map 
\[
\Fun(T,\calC)\to \Fun(S,\calC) 
\]
is an equivalence of $\infty$-categories.  
\end{definition} 

\begin{remark} 
\label{rem:2 out of 3 cat equiv}
Clearly the class of categorical equivalences 
in $\sSet$ satisfies the 2-out-of-3 property. 
\end{remark} 

The following lemma is proved by a 
straightforward adjointness 
argument.  

\begin{lemma} 
If $f\colon \calC\to \calD$ is an equivalence 
between $\infty$-categories then 
$f$ is a categorical equivalence.  
\end{lemma} 

\begin{lemma} 
If $f\colon S\to T$ is inner anodyne then $f$ is a categorical 
equivalence.  
\end{lemma} 

\begin{proof} 
This follows immediately from 
Corollary 2.3.2.5 of \cite{HTT}. 
\end{proof} 

The following lemma is well-known 
and easy to prove.  

\begin{lemma} 
If $p\colon X\to Y$ is a trivial Kan fibration then 
$p$ is a categorical equivalence.  
\end{lemma}

\begin{notation} 
We recall the following notation from 
\cite{JT}: if $\calC$ is an $\infty$-category 
and $K$ is a simplicial set then we write 
$\calC^{(K)}$ for the {\em full} simplicial subset 
(see Remark~\ref{rem:full simplicial subset} below) 
of the $\infty$-category $\calC^K$ whose vertices 
are the maps $K\to \calC$ which factor 
through the maximal Kan complex $\calC^{\simeq}$.  
\end{notation} 

\begin{remark}
\label{rem:full simplicial subset}
Here we recall 
that a full simplicial subset $T\subseteq S$ 
of a simplicial set $S$ 
is determined by the subset $T_0\subseteq 
S_0$ of $0$-simplices, in the sense 
that there is a pullback diagram 
\[
\begin{tikzcd} 
T \arrow[d] \arrow[r] & 
S \arrow[d]            \\ 
\cosk_0T_0 \arrow[r,hook] 
& \cosk_0S 
\end{tikzcd} 
\]
Clearly every full simplicial 
subset $T\subseteq S$ is an 
inner fibration.     
\end{remark} 

The following example is a 
special case of a more general 
result due to Joyal (see Theorem 5.10 
of \cite{J2}).  

\begin{example} 
\label{ex:0 <J induces tkf}
Let $\calC$ be an $\infty$-category.  
Then the inclusion $\set{0}\subset \Delta^1$ 
induces a trivial Kan fibration 
$\calC^{(\Delta^1)}\to \calC^{\set{0}}$.  
To see this it 
suffices to prove that the 
indicated diagonal filler 
exists in every commutative 
diagram of the form 
\[
\begin{tikzcd} 
\partial\Delta^n \arrow[r] \arrow[d] &
\calC^{(\Delta^1)} \arrow[d]         \\ 
\Delta^n \arrow[r] 
\arrow[ur,dashed] & 
\calC^{\set{0}} 
\end{tikzcd} 
\]
for $n\geq 1$ (the existence of such 
fillers being clear when $n=0$).  
By adjointness, the existence of 
such a diagonal filler is equivalent 
to the existence of an extension 
of the induced map $f\colon \partial\Delta^n 
\times \Delta^1 \cup \Delta^n\times 
\set{0}\to \calC$ along the inclusion 
$\partial\Delta^n\times \Delta^1 \cup 
\Delta^n\times\set{0}\subseteq 
\Delta^n\times \Delta^1$.  By assumption, 
for every vertex $i$ of $\Delta^n$, 
the restriction $f\colon \set{i}\times \Delta^1 
\to \calC$ is an equivalence in the $\infty$-category 
$\calC$.  It is a well known fact 
(see for instance the proof of 
Proposition 2.1.2.6 in \cite{HTT}) 
that there exists a sequence of 
inclusions 
\[
X(n+1) \subseteq X(n) \subseteq \cdots 
\subseteq X(1)\subseteq X(0) 
\]
with $X(n+1) = \partial\Delta^n
\times \Delta^1\cup \Delta^n\times 
\set{0}$ and $X(0) = \Delta^n
\times \Delta^1$, and where each 
$X(k)$ is obtained from $X(k+1)$ 
as a pushout of the form 
\[
X(k) = X(k+1)\cup_{\Lambda^{n+1}_k}
\Delta^{n+1}.  
\]
Since the inclusion $X(n+1)\subseteq 
X(1)$ is inner anodyne we may choose 
an extension $\tilde{f}\colon X(1)
\to \calC$ of $f$.  Since the 
edge $f\colon \set{0}\times \Delta^1
\to \calC$ is an equivalence in $\calC$ 
it follows from Theorem 1.3 of \cite{J1} 
that $\tilde{f}$ extends 
along the inclusion $X(1)\subseteq 
X(0)$.   
\end{example}

\begin{remark} 
The class of categorical equivalences in $\sSet$ is stable 
under filtered colimits. 
\end{remark}

\subsection{Dwyer-Kan equivalences} 

If $\calC$ is an $\infty$-category 
and $x,y$ are objects of $\calC$ then 
there is a mapping space $\Map_{\calC}(x,y)$ of 
morphisms from $x$ to $y$.  The space 
$\Map_{\calC}(x,y)$ is well-defined as 
an object of $\calH$, the homotopy 
category of spaces.  In \cite{HTT} 
Lurie defines several models for the homotopy 
type $[\Map_{\calC}(x,y)]$ in $\calH$.  
We shall make use of the model described 
in terms of left morphisms from $x$ to 
$y$ in $\calC$.  

Let $S$ be a simplicial set and let 
$x$ and $y$ be vertices of $S$.  Recall 
(see Section 1.2.2 of \cite{HTT}) that the   
simplicial set $\Hom^L_{S}(x,y)$ of {\em left morphisms} 
from $x$ to $y$ in $S$, is the simplicial 
set whose set of $n$-simplices is the 
set of all maps $u\colon \Delta^{n+1}\to S$ 
such that $u(0) = x$ and $u|\Delta^{\set{1,\ldots,n}}$ 
is equal to the constant $n$-simplex on the 
vertex $y$.  The face and degeneracy 
maps are induced from those of $S$ in 
the obvious way.  In other words, 
\[
\Hom^L_{S}(x,y) = S_{x/}\times_{S} \set{y}.  
\]
It is an important fact that when 
$\calC$ is an $\infty$-category, the 
simplicial set $\Hom^L_{\calC}(x,y)$ is a 
Kan complex which presents the 
mapping space $\Map_{\calC}(x,y)$ of morphisms 
in $\calC$ from $x$ to $y$.  

\begin{remark}
\label{rem:naive mapping spaces} 
One should beware that if $S$ is a simplicial 
set which presents an $\infty$-category 
$\calC$ in the sense that there is a 
categorical equivalence $S\to \calC$, 
then the naive mapping space $\Hom^L_{S}(x,y)$ 
need not have the 
homotopy type of $\Map_{\calC}(x,y)$. 
\end{remark}

\begin{definition}
\label{def:DK equivalence}
Let $\calC$ and $\calD$ be $\infty$-categories.  
A map $f\colon \calC\to \calD$ is said to be 
\begin{enumerate}[label=(\roman*)]
\item {\em essentially surjective} if the functor 
$\h(f)\colon \h(\calC)\to \h(\calD)$ is essentially 
surjective.  

\item{\em fully faithful} if the induced map 
\[
\Hom^L_{\calC}(C,C')\to \Hom^L_{\calD}(f(C),f(C')) 
\]
is a homotopy equivalence for every pair of 
vertices $C,C'\in \calC$.  
\end{enumerate} 
If $f\colon \calC\to \calD$ is fully faithful 
and essentially surjective then we will say that 
$f$ is a {\em Dwyer-Kan equivalence}.  
\end{definition} 

We will show in Proposition~\ref{prop:ff + ess surj iff equiv} 
below that a map between $\infty$-categories 
is an equivalence if and only if it is 
a Dwyer-Kan equivalence.  We begin by proving 
that every equivalence between $\infty$-categories 
is fully faithful (it is clear that every 
such equivalence is essentially surjective).

\begin{lemma} 
Let $f\colon \calC\to \calD$ be an equivalence 
between $\infty$-categories $\calC$ and $\calD$.  
Then $f$ is fully faithful.  
\end{lemma} 

\begin{proof} 
Suppose first that $f\colon \calC\to \calD$ is a 
trivial Kan fibration.  Then 
for any vertex $C\in \calC$ the induced map
\[
\calC_{C/}\to \calD_{f(C)/}\times_{\calD}\calC 
\]
is a trivial Kan fibration.  Therefore, for 
any vertex $C'\in \calC$, the map 
\[
\Hom^L_{\calC}(C,C')\to \Hom^L_{\calD}(f(C),f(C')) 
\]
is also a trivial Kan fibration.   
Hence $f$ is fully faithful in this case.  

Suppose now that $f\colon \calC\to \calD$ is 
an arbitrary equivalence between $\infty$-categories 
$\calC$ and $\calD$.  Let $g\colon \calD\to \calC$ be an 
inverse equivalence for $f$ 
as in Definition~\ref{def:equiv of oo-cats}.  
Let $h\colon \calC\to 
\calC^{J}$ be a homotopy witnessing the relation 
$gf=\id_{\calC}$ in $\h(\Fun(\calC,\calC))$; by an 
abuse of notation write $h$ also for the composite 
with the canonical projection $\calC^J \to \calC^{(\Delta^1)}$.  
We have a commutative diagram 
\[
\begin{tikzcd} 
& \calC \\ 
\calC \arrow[ur,"{\id_{\calC}}"] \arrow[r,"h"] 
\arrow[dr,"gf"'] & \calC^{(\Delta^1)} \arrow[u] \arrow[d]       \\
& \calC  
\end{tikzcd} 
\]
in which the vertical arrows are trivial Kan fibrations 
by Example~\ref{ex:0 <J induces tkf}.  
It follows from the discussion above that $gf$ is fully faithful.  
By symmetry, $fg$ is fully faithful.    

Let $C,C'\in \calC$ be vertices.  We have an induced 
diagram 
\[
\begin{tikzcd} 
\Hom^L_{\calC}(C,C') \arrow[r,"f"] \arrow[d] & 
\Hom^L_{\calD}(f(C),f(C')) \arrow[dl,"g"'] \arrow[d] \\ 
\Hom^L_{\calC}(gf(C),gf(C')) \arrow[r,"f"] 
& \Hom^L_{\calD}(fgf(C),fgf(C')) 
\end{tikzcd} 
\]
The vertical arrows are homotopy equivalences since 
$gf$ and $fg$ are fully faithful.  Therefore, by the 
2-out-of-6 property for homotopy equivalences, it 
follows that the map 
\[
\Hom^L_{\calC}(C,C')\to \Hom^L_{\calD}(f(C),f(C')) 
\]
is a homotopy equivalence.  Hence $f$ is fully faithful. 
\end{proof}

\begin{remark} 
More generally we can make sense of what 
it means for a map of simplicial sets to 
be essentially surjective or fully faithful 
(the above definitions are invariant under 
equivalences of $\infty$-categories).  
\end{remark} 

We would now like to prove that every 
Dwyer-Kan equivalence $f\colon \calC\to \calD$  
is a categorical equivalence.  
For this we shall need to isolate 
a particular class of inner fibrations 
between $\infty$-categories.  
With the benefit of hindsight 
(see Corollary 2.4.6.5 of \cite{HTT})
we make the following definition.     

\begin{definition} 
We will say that an inner fibration 
$p\colon \calC\to \calD$ 
between $\infty$-categories 
is a {\em categorical fibration} if the functor 
$\h(p)\colon \h(\calC)\to \h(\calD)$ 
is an isofibration of categories.  
\end{definition} 

\begin{remark}
An inner fibration $p\colon \calC
\to \calD$ between $\infty$-categories 
is a categorical fibration if and only if it has the 
right lifting property with respect 
to the inclusion $\set{0}\subseteq J$, 
if and only if the canonical map 
$\calC^{(\Delta^1)}\to 
\calC\times_{\calD}\calD^{(\Delta^1)}$ 
induced by the inclusion 
$\Delta^{\set{0}}\subseteq \Delta^1$ 
is surjective.  
\end{remark} 

\begin{lemma}[Joyal] 
\label{lem:ps fibn + ce implies tkf}
Suppose $p\colon \calC\to \calD$ is a 
categorical fibration between $\infty$-categories.  
If $p$ is fully faithful and 
essentially surjective then 
$p$ is a trivial Kan fibration.  
\end{lemma} 

\begin{proof} 
Suppose $p$ is fully 
faithful and essentially surjective.  
Since $p$ has the right lifting property 
against the inclusion $\set{0}\subseteq J$ 
it follows that $p$ is surjective.  

We prove that for every $n\geq 1$ the indicated 
diagonal filler exists in every commutative 
diagram of the form 
\[
\begin{tikzcd} 
\partial\Delta^n\arrow[d] \arrow[r,"u"] & 
\calC \arrow[d,"p"]                      \\ 
\Delta^n\arrow[r,"v"] & \calD.  
\end{tikzcd} 
\]
By an adjointness argument this is 
equivalent to proving that the 
indicated diagonal filler exists 
in the induced diagram 
\begin{equation}
\label{eq:induced diag1}
\begin{tikzcd} 
\partial\Delta^{\set{1,\ldots, n}} 
\arrow[r] \arrow[d] & \calC_{u(0)/} \arrow[d] \\ 
\Delta^{\set{1,\ldots, n}} \arrow[r] \arrow[ur,dashed]
& \calD_{v(0)/}\times_{\calD} \calC 
\end{tikzcd} 
\end{equation} 
The induced map 
\begin{equation}
\label{eq:slice left fibn} 
\calC_{u(0)/}\to \calD_{v(0)/}\times_{\calD}\calC 
\end{equation} 
is a left fibration (Proposition 2.1.2.1 of \cite{HTT}).  
Since $p$ is fully faithful 
it follows that the fibers 
of~\eqref{eq:slice left fibn} are contractible.  
Hence~\eqref{eq:slice left fibn} is a trivial 
Kan fibration (Lemma 2.1.3.4 of \cite{HTT}).  
Hence the indicated lift 
exists in the diagram ~\eqref{eq:induced diag1}.  
It follows that $p$ is a 
trivial Kan fibration.  
\end{proof}

The following example appears 
as Proposition 5.16 in 
\cite{J2}.  The labored account 
that we give of it here is 
an unfortunate consequence 
of our efforts to keep the paper 
self-contained.

\begin{example}  
Let $\calC$ be an $\infty$-category.  The 
canonical map $\calC^{(\Delta^1)}\to \calC\times \calC$ 
induced by the inclusion $\set{0,1}\subseteq 
\Delta^1$ is an inner fibration by Corollary 
2.3.2.5 of \cite{HTT}.  In fact this canonical map 
is also a categorical fibration.  
To see this, suppose given a 
commutative diagram of the 
form 
\[
\begin{tikzcd} 
\set{0} \arrow[r,"u"] \arrow[d] & \calC^{(\Delta^1)} \arrow[d] \\ 
\Delta^1 \arrow[ur,dashed] \arrow[r,"{(f,g)}"'] & 
\calC\times \calC 
\end{tikzcd} 
\] 
in which the edges $f$ and $g$ are 
equivalences in the $\infty$-category $\calC$.  
We must prove that the indicated diagonal 
filler exists.  The existence of such a 
diagonal filler is equivalent to the 
existence of a map $\phi\colon \Delta^1\times 
\Delta^1\to \calC$ such that 
the restriction $\phi|\Delta^{\set{0}}
\times \Delta^1$ is equal to $u$, 
the restriction $\phi|\Delta^1\times 
\Delta^{\set{0}}$ is equal to $f$, and 
the restriction $\phi|\Delta^1\times \Delta^{\set{1}}$ is 
equal to $g$.  The existence of such a map 
$\phi$ is easily proven using Proposition 1.2.4.3 
of \cite{HTT} (notice the restriction 
$\phi|\Delta^{\set{1}}\times \Delta^1$ 
is necessarily an equivalence in $\calC$, so that $\phi$ 
corresponds to a map $\Delta^1\to \calC^{(\Delta^1)}$).  
\end{example} 


\begin{example} 
\label{ex:mapping path space}
(Mapping path space construction).
Given a map $f\colon \calC\to \calD$ 
between $\infty$-categories 
$\calC$ and $\calD$ we define the 
$\infty$-category $Q(f)$ by the 
pullback diagram 
\[
\begin{tikzcd} 
Q(f) \arrow[r] \arrow[d] 
& \calC\times \calD \arrow[d,"{f\times \id_{\calD}}"] \\ 
\calD^{(\Delta^1)} \arrow[r] & 
\calD\times \calD 
\end{tikzcd} 
\]
It follows from the discussion 
in the paragraph above that 
the composite map 
$Q(f)\to \calC\times \calD 
\to \calD$ is a categorical fibration.  
Denote 
this composite map by 
$\pi\colon Q(f)\to \calD$. 
The diagonal map $\calD\to \calD^{(\Delta^1)}$ 
induces a map $i\colon \calC\to Q(f)$ which 
is right inverse to the trivial 
Kan fibration given by the composite 
map $Q(f)\to \calC\times \calD\to \calC$.  
It follows that $i$ is a 
categorical equivalence.  
The factorization $f = \pi i$ is called 
the {\em mapping path space} factorization 
of $f$.  
\end{example} 

\begin{proposition}[Joyal] 
\label{prop:ff + ess surj iff equiv}
Let $\calC$ and $\calD$ be $\infty$-categories.  
A map $f\colon \calC\to \calD$ is an equivalence 
if and only if it is a Dwyer-Kan 
equivalence.  
\end{proposition} 

\begin{proof} 
We have seen above that every equivalence 
between $\infty$-categories is fully 
faithful and essentially surjective.  
We need to prove 
the converse.  Suppose that $f\colon 
\calC\to \calD$ is fully faithful and 
essentially surjective, where $\calC$ 
and $\calD$ are $\infty$-categories.  
Example~\ref{ex:mapping path space} 
gives a factorization of the map 
$f\colon \calC\to \calD$ as 
\[
\calC \xrightarrow{i} Q(f)\xrightarrow{p} \calD 
\]
where $i$ is right inverse to a trivial 
Kan fibration and where $p$ is a  
categorical fibration.  
It follows that $i$ is fully faithful 
and essentially surjective.  
Therefore we may suppose without loss 
of generality that the map $f\colon \calC\to 
\calD$ is a categorical fibration.
Lemma~\ref{lem:ps fibn + ce implies tkf} 
then shows that $f$ is a trivial 
Kan fibration.       
\end{proof} 

\begin{lemma} 
\label{lem:cat equivs stable under po}
For any pushout diagram 
\[
\begin{tikzcd} 
A \arrow[r] \arrow[d,"i"] & C \arrow[d,"j"] \\ 
B \arrow[r] & D 
\end{tikzcd} 
\]
of simplicial sets in which 
$i$ is a monic categorical equivalence then so is $j$.  
\end{lemma} 

\begin{proof} 
Since trivial Kan fibrations 
are stable under base change it suffices 
to prove that for any $\infty$-category 
$\calC$, the map 
$\calC^{B}\to \calC^{A}$ 
induced by $i$ is a trivial 
Kan fibration.  The induced map 
is a categorical equivalence 
by hypothesis, and is an inner fibration 
by Corollary 2.3.2.5 of \cite{HTT}.  
Therefore it suffices by 
Lemma~\ref{lem:ps fibn + ce implies tkf} 
to prove that this inner fibration is 
a categorical fibration.  
Suppose given a commutative diagram 
\[
\begin{tikzcd} 
\Delta^{\set{0}} 
\arrow[r] \arrow[d] & 
\calC^B \arrow[d]  \\ 
\Delta^1 \arrow[r,"f"] 
\arrow[ur,dashed] & 
\calC^A 
\end{tikzcd} 
\]
in which the edge $f\colon 
\Delta^1\to \calC^A$ is an 
equivalence.  We need to 
prove that the indicated 
diagonal filler exists.  
Choose an extension 
$g\colon J\to \calC^A$ of $f$ 
along the map $\Delta^1\to J$.  
Observe that a diagonal 
filler exists in the 
diagram above if and only if 
the indicated diagonal filler 
exists in the diagram 
\[
\begin{tikzcd} 
\set{0} \arrow[r] \arrow[d] 
& \calC^B \arrow[d]         \\ 
J \arrow[r,"g"] \arrow[ur,dashed] 
& \calC^A 
\end{tikzcd}
\]
Denote by $\overline{f}\colon 
A\to \calC^{\Delta^1}$ and 
$\overline{g}\colon A\to 
\calC^J$ the maps conjugate to 
$f$ and $g$ respectively.  
Since $f$ is equal to the 
composite $\Delta^1\to J 
\xrightarrow{g} \calC^A$, the 
map $\overline{f}$ is equal 
to the composite $A\xrightarrow{ 
\overline{g}} \calC^J\to \calC^{\Delta^1}$.  
In particular it follows 
that $\overline{f}$ factors through 
$\calC^{(\Delta^1)}$ via a unique 
map $k\colon A\to \calC^{(\Delta^1)}$.  
Since the composite map 
$\calC^{(\Delta^1)}\to \calC^{\Delta^1} 
\to \calC^{\set{0}}$ is a trivial 
Kan fibration (Example~\ref{ex:0 <J induces tkf}), 
it follows that there is a map 
$B\to \calC^{(\Delta^1)}$ whose 
composite with $i\colon A\to B$ 
is equal to the map $k$.  It 
follows that the indicated diagonal 
filler in the diagram 
\[
\begin{tikzcd} 
A \arrow[r,"\overline{f}"] \arrow[d] & 
\calC^{\Delta^1} \arrow[d]             \\  
B \arrow[ur,dashed] \arrow[r] & 
\calC^{\set{0}} 
\end{tikzcd} 
\]
exists.  Therefore, by adjointness, 
we can find a diagonal filler $\hat{f}$ 
for the diagram 
\[
\begin{tikzcd} 
\set{0} \arrow[r] \arrow[d] & \calC^B \arrow[d] \\ 
\Delta^1 \arrow[ur,"\hat{f}"] \arrow[r] & 
\calC^A 
\end{tikzcd} 
\]
Therefore, we have proven that for 
every equivalence $f$ in $\h(\calC^A)$ 
and every vertex $u$ in $\h(\calC^B)$ such 
that $ui = f(0)$, there is an edge $\hat{f}$ 
in $\h(\calC^B)$ which projects to $f$.  
Since $\h(\calC^B)\to \h(\calC^A)$ is 
fully faithful, it follows that 
$\h(\calC^B)\to \h(\calC^A)$ is an 
isofibration.    
\end{proof} 

\begin{remark} 
Of course the above Lemma is an 
immediate consequence of the existence 
of the Joyal model structure on $\sSet$.  
One of our aims is to give an independent 
proof of the existence of this model structure; 
in order to avoid a circular argument 
(see the proof of Proposition~\ref{prop:key tech prop2}) 
we will need to use this lemma.  
\end{remark} 

\section{pre-fibrant simplicial sets} 
\label{sec:pre-fibrant simplicial sets} 

We have observed above that 
if $S$ is a simplicial set 
and $x,y$ are vertices of $S$ 
then the naive mapping space 
$\Hom^L_{S}(x,y)$ may not 
have the homotopy type of 
the mapping space $\Map_{\calC}(x,y)$ 
for an $\infty$-category $\calC$ 
generated by $S$.  
In this section we give a 
condition on a simplicial set 
to ensure that the naive mapping 
spaces have the correct homotopy types.  

\subsection{Pre-fibrant simplicial sets} 
We introduce the following definition.  

\begin{definition}
We will say that  
simplicial set $S$ is {\em pre-fibrant} 
if it satisfies the following conditions: 
\begin{enumerate}[label=(\roman*)]
\item every map $\Lambda^2_1\to S$ extends 
along the inclusion $\Lambda^2_1\subseteq \Delta^2$.  

\item for every $0<i<n$, if $f\colon \Lambda^n_i \to S$ 
is a map such that $d_0f$ is a constant 
$(n-1)$-simplex, 
then $f$ extends along 
the inclusion $\Lambda^n_i\subseteq \Delta^n$.  
\end{enumerate} 
\end{definition} 

Here we will say that an $n$-simplex $x\in S$ is {\em constant} 
if $x = s_0^{n}(v)$ for some vertex $v$ of $S$.  

\begin{example} 
Every $\infty$-category is a pre-fibrant simplicial set.  
\end{example} 

\begin{remark} 
If $S$ is a pre-fibrant simplicial set then 
for every pair of vertices $x,y\in S$, 
the naive mapping space $\Hom^L_{S}(x,y)$ 
is a Kan complex.  
\end{remark} 

\begin{remark} 
\label{rem:condition *}
Note that if $p\colon X\to S$ is an inner fibration, 
and $S$ is pre-fibrant, then $X$ is also 
pre-fibrant.  In particular if $S$ is pre-fibrant and 
$T\subseteq S$ is a full simplicial 
subset (see Remark~\ref{rem:full simplicial subset}) 
then $T$ is pre-fibrant.  
\end{remark} 

\subsection{Pre-fibrant simplicial sets and $\infty$-categories}
In this section we prove that the naive mapping 
spaces for a pre-fibrant simplicial set $S$ compute the mapping 
spaces of an $\infty$-category generated by $S$.  
More precisely, we prove the following result.

\begin{proposition} 
\label{prop:key prop}
Suppose that $S$ is a pre-fibrant simplicial set.  
Then there exists an inner anodyne map 
$f\colon S\to T$ with the following properties: 

\begin{enumerate} 
\item $T$ is an $\infty$-category; 
\item the induced map 
\[
S_{x/}\times_S \set{y}\to T_{x/}\times_T \set{y} 
\]
is an isomorphism for every pair of vertices 
$x,y\in S$.  
\end{enumerate}
\end{proposition} 

\begin{proof} 
We construct a simplicial set $T$ which contains 
$S$ as a subcomplex as follows.    
Define $\sk_3T$ by the pushout diagram  
\[
\begin{tikzcd} 
\bigsqcup_{a\in A_3}\Lambda^3_{i_a} \ar[d] \ar[r] & \sk_3S \ar[d] \\  
\bigsqcup_{a\in A_3}\Delta^3 \ar[r] & \sk_3 T
\end{tikzcd}
\]
where $A_3$ denotes the set of maps 
\[
a\colon \Lambda^3_i\to \sk_3 S 
\]
with $0<i<3$ and where $d_0(a)$ is non-constant. 
Clearly the canonical map $\sk_3S \to \sk_3T$ is inner anodyne.  
Note the following facts: (i) if $\sigma\colon \Delta^3\to \sk_3T$ does 
not belong to $\sk_3S$, then $d_0(\sigma)$ is 
non-constant, and (ii) $\sk_1 S = \sk_1 T$.   

Assume inductively that $\sk_nT$ has been constructed for $n\geq 3$, containing 
$\sk_nS$ as a subcomplex and with the following properties: 

\begin{enumerate}[label={(P-\arabic*)}] 
\item the inclusion $\sk_n S\subseteq \sk_nT$ is inner anodyne; 
\item if $\sigma\colon \Delta^n\to \sk_nT$ is an $n$-simplex which 
does not belong to $\sk_nS$ then $d_0(\sigma)$ is 
non-constant. 
\end{enumerate} 

Define $\sk_{n+1}T$ by the pushout diagram 
\begin{equation} 
\label{eq:pushout diagram for n skeleton}
\begin{tikzcd} 
\bigsqcup_{a\in A_{n+1}} \Lambda^{n+1}_{i_a} \ar[r] \ar[d] & \sk_{n+1}S\cup \sk_n T \ar[d] \\ 
\bigsqcup_{a\in A_{n+1}} \Delta^{n+1} \ar[r] & \sk_{n+1}T 
\end{tikzcd} 
\end{equation}
where $A_{n+1}$ denotes the set of maps 
\[
a\colon \Lambda^{n+1}_i\to \sk_{n+1}S\cup \sk_nT 
\]
with $0<i<n+1$ and where $d_0(a)$ is non-constant.

To close 
the inductive loop we need to prove that 
$\sk_{n+1}T$ satisfies the corresponding 
properties (P-1) and (P-2).  
For (P-1), we need to prove that the canonical map 
$\sk_{n+1}S\to \sk_{n+1}T$ is inner anodyne.   
For this, note that the  
map $\sk_{n+1}S\cup \sk_nT\to \sk_{n+1}T$ 
in the diagram~\eqref{eq:pushout diagram for n skeleton} above is inner anodyne.  Since the map 
\[
\sk_{n+1}S\to \sk_{n+1}S\cup \sk_nT 
\]
is a pushout of the inner anodyne map $\sk_nS\to \sk_nT$, it follows 
that the composite map $\sk_{n+1}S\to \sk_{n+1}T$ is inner anodyne.
For (P-2), we need to prove that if 
$\sigma\colon \Delta^{n+1}\to \sk_{n+1}T$ is an 
$(n+1)$-simplex which does not belong to 
$\sk_{n+1}S$ then $d_0(\sigma)$ is non-constant.  If 
$\sigma$ is non-degenerate then $d_0(\sigma)$ is non-constant by 
construction.  Suppose then that $\sigma$ is degenerate --- suppose 
$\sigma = s_i(x)$ where $x\in \sk_nT$.    
Suppose for a contradiction that $d_0(\sigma)$ is constant --- suppose 
$d_0(\sigma) = s_0^{n}(v)$ for some vertex $v$ of $S$.  
If $i=0$ then $x = s_0^{n}(v)$; if $i\geq 1$ then 
\begin{align*} 
d_0(x) & = d_0d_is_i(x) \\ 
& = d_0d_i(\sigma) \\ 
& = d_{i-1}s_0^{n}(v) \\ 
& = s_0^{n-1}(v).  
\end{align*} 
In either case we see that $d_0(x)$ is constant.  This is a contradiction 
to the inductive hypothesis, since $\sigma$ 
does not belong to $\sk_{n+1}S$ and 
hence $x$ does not belong to $\sk_nS$.  Therefore $d_0(\sigma)$ is 
not constant.   

This completes the inductive step and hence the construction of the 
simplicial set $T$ containing $S$ as a subcomplex.  Since each inclusion  
$\sk_nS\subseteq \sk_nT$, $n\geq 3$, is inner anodyne, it follows from 
Lemma~\ref{lem:ladders of inner anodynes} that the inclusion $S\subseteq T$ is 
inner anodyne.  

It remains to prove that 
\begin{enumerate}  
\item the induced map 
\[
S_{x/}\times_S \set{y}\to T_{x/}\times_T \set{y} 
\]
is an isomorphism for every pair of vertices $x,y\in S$;
\item $T$ is an $\infty$-category.  
\end{enumerate} 
We prove the first statement.  Clearly we have 
an injective map  
\[
S_{x/}\times_{S} \set{y} \to T_{x/}\times_{T} \set{y}  
\]
since $S\to T$ is a monomorphism.
Let $\sigma\colon \Delta^n\to 
T_{x/}\times_T \set{y}$ be an $n$-simplex.  Then 
$\sigma$ corresponds to an $(n+1)$-simplex $\tilde{\sigma}
\colon \Delta^{n+1}\to T$ such that $d_0(\tilde{\sigma})$ is the 
constant $n$-simplex at $y$ and $\tilde{\sigma}|\Delta^{\set{0}} = x$. 
Since $d_0(\tilde{\sigma})$ is constant, 
$\tilde{\sigma}$ must belong to 
$\sk_{n+1}S$ by construction.  Hence $\tilde{\sigma}$ factors 
through $S$ and hence the $n$-simplex $\sigma\colon \Delta^n
\to T_{x/}\times_T\set{y}$ factors through $S_{x/}\times_S \set{y}$.  
Hence we have an isomorphism 
\[
S_{x/}\times_S \set{y} \simeq T_{x/}\times_T \set{y} .  
\]
We prove the second statement, i.e.\ we prove that 
$T$ is an $\infty$-category.  Suppose given a map 
\[
f\colon \Lambda^n_i\to T 
\]
with $0<i<n$ and where $n\geq 3$.  Then $f$ factors 
through $\sk_{n-1}T$.  If $d_0f$ is non-constant, 
then the composite map 
$\Lambda^n_i\xrightarrow{f}\sk_{n-1}T\to \sk_{n}T$ extends 
along the inclusion $\Lambda^n_i\subseteq \Delta^n$.  

Suppose now that $d_0f$ is constant.  
Then for every face 
$\partial_j\Delta^{n-1}$ with $j\neq i$, the 
induced map $d_jf := f|\partial_j\Delta^{n-1}$ is an 
$(n-1)$-simplex $d_jf\colon \Delta^{n-1}\to \sk_{n-1}T$ 
with $d_0(d_jf)$ constant.  
Hence each map 
$d_jf$ factors through $\sk_{n-1}S$.  Hence $f$ 
factors through $\sk_{n-1}S$.  By hypothesis, 
the composite map $\Lambda^n_i\xrightarrow{f} \sk_{n-1}S 
\to \sk_nS$ extends along the inclusion 
$\Lambda^n_i\subseteq \Delta^n$.  

Finally, since $\sk_1 S = \sk_1 T$ and $S$ 
has the right lifting property against the 
inclusion $\Lambda^2_1\subseteq \Delta^2$, it 
follows that $T$ is an $\infty$-category.     
\end{proof} 

\begin{remark} 
\label{rem:mapping space for condition star}
Suppose that $S$ is a pre-fibrant simplicial set.  
If $u\colon S\to T$ is an inner anodyne map where 
$T$ is an $\infty$-category then the induced map 
\[
\bar{u}\colon S_{s/}\times_{S} \set{s'}\to 
T_{u(s)/}\times_{T}\set{u(s')} 
\]
is a homotopy equivalence between Kan complexes for 
any pair of vertices $s,s'\in S$.  To see this, 
choose an inner anodyne map $v\colon S\to T'$ 
with the properties described in 
Proposition~\ref{prop:key prop}.  Then there 
exists a map $w\colon T\to T'$ such that $v = wu$.  
The statement follows by the 2-out-of-3 property 
for weak homotopy equivalences.     
\end{remark} 

\begin{remark} 
\label{rem:fully faithful for pre-fibrant simp set}
It follows from Proposition~\ref{prop:key prop} 
that a map $f\colon S\to T$ of pre-fibrant simplicial sets 
is fully faithful if and only if the induced map 
$\Hom^L_{S}(x,y)\to \Hom^L_{T}(f(x),f(y))$ on 
the naive mapping spaces is a homotopy 
equivalence between Kan complexes 
for all objects $x$ and $y$ of $S$. 
\end{remark} 

\subsection{The homotopy category of a pre-fibrant simplicial set} 
In this section we make some 
observations about the homotopy 
category of pre-fibrant simplicial sets 
that we shall put to use later. 

\begin{lemma} 
\label{corr:hty cat of pre-fibrant sset}
Let $S$ be a pre-fibrant simplicial set.  Then for 
any pair of vertices $s,s'\in S$, there 
is an isomorphism 
\[
\pi_0(S_{s/}\times_{S}\set{s'}) 
\simeq \h(S)(s,s'). 
\]
\end{lemma} 

\begin{proof} 
Let $u\colon S\to T$ be an inner anodyne 
map as in Proposition~\ref{prop:key prop}.  
The functor $\h(S)\to \h(T)$ induced 
by $u$ is an isomorphism.    
For any pair of vertices $s,s'\in S$ 
we have an isomorphism 
\[
\pi_0(S_{s/}\times_{S}\set{s'})\simeq 
\pi_0(T_{u(s)/}\times_{T}\set{u(s')}). 
\]
But $\pi_0(T_{u(s)/}\times_{T}\set{u(s')}) = 
\pi(T)(u(s),u(s'))$ and we have an 
isomorphism $\h(T)\simeq \pi(T)$ (here $\pi(T)$ 
denotes the category constructed in Section 1.2.3 
of \cite{HTT}).  
\end{proof} 

\begin{lemma} 
\label{lem:char of isofibns}
Let $p\colon X\to S$ be an inner fibration between pre-fibrant 
simplicial sets.  Then $\h(p)\colon \h(X)\to \h(S)$ is 
an isofibration if and only if for every equivalence 
$f\colon \Delta^1\to S$ and vertex $x\in X$ such that 
$p(x) = f(0)$, there exists 
an equivalence $u\colon \Delta^1\to X$ such 
that $p(u) = f$.  
\end{lemma} 

\begin{proof} 
Suppose that $\h(p)\colon \h(X)\to \h(S)$ is an isofibration.  
Choose an inner anodyne map $S\to \calD$ with the properties 
of Proposition~\ref{prop:key prop} so that $\calD$ 
is an $\infty$-category.  The composite map $X\to S\to \calD$ 
factorizes as $X\to \calC\to \calD$ where $X\to \calC$ is inner 
anodyne and $\calC\to \calD$ is an inner fibration.  Thus we have 
a commutative diagram of the form 
\[
\begin{tikzcd} 
X \arrow[r] \arrow[d] & \calC\arrow[d] \\ 
S \arrow[r] & \calD 
\end{tikzcd} 
\]
Since $\h(p)\colon \h(X)\to \h(S)$ is an isofibration it follows 
that the induced map $\h(\calC)\to \h(\calD)$ is an isofibration.  
Hence there exists an equivalence $v\colon \Delta^1\to \calC$ 
such that $v(0) = x$ and $p(v) = f$.  Let $y = v(1)$ and $t = f(1)$.  We have an 
induced map 
\[
\begin{tikzcd} 
X_{x/}\times_{X} \set{y} \arrow[r] \arrow[d] & 
\calC_{x/}\times_{\calC}\set{y} \arrow[d] \\ 
S_{s/}\times_{S} \set{t} \arrow[r] & 
\calD_{s}\times_{\calD} \set{t} 
\end{tikzcd} 
\]
where the lower horizontal map is an isomorphism 
(Proposition~\ref{prop:key prop}), the vertical maps 
are Kan fibrations between Kan complexes, and the upper horizontal map is 
a homotopy equivalence (Remark~\ref{rem:mapping space for condition star}).
It follows that we can choose a vertex $u$ in the fiber of $X_{x/}\times_{X}\set{y}
\to S_{s/}\times_{S}\set{t}$ over $f$ whose image in $\calC_{x/}\times_{\calC} 
\set{y}$ is homotopic to $v$.  Thus $u$ represents a morphism in $\h(X)(x,y)$ 
whose image in $\h(\calC)(x,y)$ is equal to $[v]$.  Since $\h(X)\to \h(\calC)$ 
is an isomorphism it follows that $[u]$ is an isomorphism in $\h(X)$.  
The proof of the converse statement is straightforward and is left 
to the reader.  
\end{proof} 

\begin{remark} 
\label{rem:subcomplex S'}
If $S$ is a pre-fibrant simplicial set and 
$f\colon x\to y$ is an equivalence in 
$S$ (Section~\ref{subsec:hty cats}) 
then it is not necessarily true that 
the map $f\colon \Delta^1\to S$ 
classifying $f$ extends along the canonical 
map $\Delta^1\to J$.  However, it 
follows from Lemma~\ref{corr:hty cat of pre-fibrant sset} 
that there exists an edge $g\colon y\to x$ 
such that $[g][f]= \id_x$ and 
$[f][g] = \id_y$ in $\h(S)(x,x)$ 
and $\h(S)(y,y)$ respectively.  
Hence there exist $2$-simplices 
$\sigma,\sigma'\colon \Delta^2\to S$ 
corresponding to diagrams 
\[
\begin{tikzcd} 
& y \arrow[dr,"g"] & \\ 
x \arrow[ur,"f"] \arrow[rr,"\phi"] 
& & x 
\end{tikzcd} 
\quad \text{and}\quad 
\begin{tikzcd} 
& x \arrow[dr,"f"] & \\ 
y \arrow[ur,"g"] \arrow[rr,"\psi"] 
& & y 
\end{tikzcd}
\]
respectively.  Moreover $[\phi] = \id_x$ 
in $\h(S)(x,x)$ and $[\psi] = \id_y$ in 
$\h(S)(y,y)$.  Therefore, 
using Lemma~\ref{corr:hty cat of pre-fibrant sset} 
again, we see that, since $S_{x/}\times_{S} \set{x}$ 
and $S_{y/}\times_{S}\set{y}$ 
are Kan complexes, there exist $2$-simplices 
$\tau,\tau'\colon \Delta^2\to S$ 
corresponding to diagrams 
\[
\begin{tikzcd} 
& x \arrow[dr,"\id_x"] & \\ 
x \arrow[ur,"\phi"] \arrow[rr,"\id_x"] 
& & x 
\end{tikzcd} 
\quad \text{and}\quad 
\begin{tikzcd} 
& y \arrow[dr,"\id_y"] & \\ 
y \arrow[ur,"\psi"] \arrow[rr,"\id_y"] 
& & y 
\end{tikzcd}
\]
respectively.
Notice that if $S'\subseteq S$ 
denotes the subcomplex generated by the 
$2$-simplices $\sigma$, $\sigma'$, $\tau$ 
and $\tau'$ then $\h(S') = J$.    
\end{remark}

\begin{lemma} 
\label{lem:pullback lemma}
Suppose given a pullback diagram 
\[
\begin{tikzcd} 
X \arrow[d,"p"] \arrow[r,"u"] & Y \arrow[d,"q"] \\ 
S \arrow[r,"v"] & T 
\end{tikzcd} 
\]
of pre-fibrant simplicial sets.  If $q\colon Y\to T$ 
is an inner fibration, then the induced 
diagram 
\[
\begin{tikzcd} 
\h(X) \arrow[r] \arrow[d] & \h(Y) \arrow[d] \\ 
\h(S) \arrow[r] & \h(T) 
\end{tikzcd} 
\]
is a pullback diagram in $\Cat$.  
\end{lemma} 

\begin{proof} 
It suffices to prove that for every pair 
of vertices $x,x'\in X$, the induced 
diagram 
\[
\begin{tikzcd} 
\h(X)(x,x') \arrow[r] \arrow[d] & \h(Y)(u(x),u(x')) \arrow[d] \\ 
\h(S)(s,s') \arrow[r] & \h(T)(v(s),v(s'))  
\end{tikzcd} 
\]
is a pullback, where $s = p(x)$ and $s' = p(x')$.  The diagram 
\[
\begin{tikzcd} 
X_{x/}\times_{X} \set{x'} \arrow[d,"p"] 
\arrow[r] & Y_{u(x)/}\times_{Y} \set{u(x')} \arrow[d,"q"] \\ 
S_{s/}\times_{S} \set{s'} \arrow[r] & T_{v(s)/}\times_{T} \set{v(s')} 
\end{tikzcd} 
\]
is a homotopy pullback.  The result 
follows by applying $\pi_0$.  
\end{proof}

\section{Descent for inner fibrations} 
\label{sec:descent}

In this section we introduce a 
notion of descent for inner fibrations.  
The case that will be of most 
interest to us is the case of 
descent along an inner anodyne 
map constructed via the small 
object argument.  Our main 
result is Proposition~\ref{prop:descent pre-fibrant simp sets}.   

\begin{definition} Suppose that $p\colon X\to S$ is an 
inner fibration between simplicial sets 
and that $f\colon S\to T$ is a map 
in $\sSet$.  We will say that $p\colon X\to S$ satisfies 
{\em descent} relative to $f\colon S\to T$ 
if there exists an inner fibration 
$q\colon Y\to T$ forming part 
of a pullback diagram 
\[
\begin{tikzcd} 
X \arrow[d,"p"] \arrow[r,"g"] & Y \arrow[d,"q"] \\ 
S \arrow[r,"f"] & T 
\end{tikzcd} 
\]
If $f\colon S\to T$ is inner 
anodyne then we will say that 
$p\colon X\to S$ satisfies {\em inner 
anodyne descent} with respect to $f$ 
if the map $g$ in the diagram above 
is inner anodyne.  
\end{definition} 

We will be interested in finding conditions 
on the inner fibration $p$ which ensure that 
$p$ satisfies descent relative to certain 
maps $f\colon S\to T$.  
The following lemma gives a convenient sufficient condition  
for an inner fibration 
to satisfy descent.  

\begin{lemma}
\label{lem:joyals descent lemma}
Suppose that $f\colon T\to S$ and $p\colon X\to S$ are maps 
of simplicial sets and that $f$ is surjective.  
If there is a pullback diagram 
\[
\begin{tikzcd} 
Y \arrow[d,"q"] \arrow[r] & X \arrow[d,"p"] \\ 
T \arrow[r,"f"] & S  
\end{tikzcd}
\]
in which $q$ is an inner fibration then $p$ is also an inner fibration.  
\end{lemma} 

\begin{proof} 
We prove that $p$ is an inner fibration.  
For this, we note that, since $f$ is surjective, 
for any commutative diagram of the form 
\begin{equation}
\label{eq:filler diagram}
\begin{tikzcd} 
\Lambda^n_k \arrow[r] \arrow[d] & X \arrow[d,"p"] \\ 
\Delta^n \arrow[r,"u"] & S,
\end{tikzcd} 
\end{equation}
the map $u$ factors as $fv$ for some map 
$v\colon \Delta^n\to T$.  It follows that the map 
$\Lambda^n_k\to X$ factors through $Y$ so 
that we have a commutative diagram 
\[
\begin{tikzcd}
\Lambda^n_k \arrow[r] \arrow[d] & Y \arrow[d,"q"] \\ 
\Delta^n \arrow[r,"v"] & T
\end{tikzcd} 
\] 
Since $q$ is an inner fibration we may 
find a diagonal filler in this diagram 
and hence in the diagram~\eqref{eq:filler diagram}.  
Therefore $p$ is an inner fibration.     
\end{proof} 

\subsection{Colimits and descent} 
In this section we give some examples 
of descent for coprojection maps into 
cocones on certain colimit diagrams.  

\begin{lemma} 
\label{lem:seq colimit descent}
Suppose given a commutative diagram of simplicial sets 
\[
\begin{tikzcd} 
X_0 \arrow[d] \arrow[r] & 
X_1 \arrow[d] \arrow[r] & 
X_2 \arrow[d] \arrow[r] & \cdots \\ 
S_0 \arrow[r] & 
S_1 \arrow[r] & 
S_2 \arrow[r] & 
\cdots 
\end{tikzcd} 
\]
in which each square is a pullback, 
all horizontal maps are monomorphisms, and all 
vertical maps are inner fibrations.  Then 
the inner fibration $X_m\to S_m$ satisfies 
descent with respect to the canonical 
map $S_m\to \varinjlim S_m$ for every 
$m\geq 0$.  
\end{lemma} 

\begin{proof} 
Let $Y = \varinjlim X_m$ and let 
$T = \varinjlim S_m$.  Under the given hypotheses, the squares 
\[
\begin{tikzcd} 
X_m \arrow[r] \arrow[d] & Y \arrow[d] \\ 
S_m \arrow[r] & T 
\end{tikzcd} 
\] 
are pullbacks for every $m$.  It follows 
that any commutative diagram of the form 
\begin{equation}
\label{eq:first square1}
\begin{tikzcd} 
\Lambda^n_i \arrow[d] \arrow[r] & Y \arrow[d] \\ 
\Delta^n \arrow[r] & T 
\end{tikzcd} 
\end{equation}
can be extended to a commutative diagram of 
the form 
\begin{equation} 
\label{eq:composite squares}
\begin{tikzcd} 
\Lambda^n_i \arrow[r] \arrow[d] & 
X_m \arrow[d] \arrow[r] & Y \arrow[d] \\ 
\Delta^n \arrow[r] & 
S_m \arrow[r] & T 
\end{tikzcd} 
\end{equation} 
for some $m\geq 0$.  A choice of diagonal 
filler for the left hand square in~\eqref{eq:composite squares} 
determines a diagonal filler for the 
square~\eqref{eq:first square1}.      
\end{proof} 

\begin{remark}
\label{rem:seq colimit descent} 
One outcome of the proof of 
Lemma~\ref{lem:seq colimit descent} is that under the 
stated hypotheses 
the canonical map 
\[
\varinjlim X_m \to \varinjlim S_m 
\]
is an inner fibration.  
\end{remark}

\begin{lemma} 
\label{lem:pushout descent}
Suppose given a commutative diagram of simplicial sets 
\[
\begin{tikzcd}
X_1 \arrow[d] & X_0 \arrow[l] \arrow[r] \arrow[d] 
& X_2 \arrow[d]                                    \\ 
S_1 & \arrow[l] S_0 \arrow[r] & S_2 
\end{tikzcd} 
\]
in which both squares are pullbacks, the arrow 
$S_0\to S_2$ is a monomorphism, and all vertical 
maps are inner fibrations.  Then for $i=1$ and $i=2$ the inner 
fibration $X_i\to S_i$ satisfies descent 
with respect to the canonical map $S_i\to S$, 
where $S = S_1\cup_{S_0}S_2$.  
\end{lemma} 

\begin{proof} 
Let $X = X_1\cup_{X_0}X_2$ and let $S = S_1\cup_{S_0}S_2$.  
Under the given hypotheses the squares 
\[
\begin{tikzcd} 
X_i \arrow[r] \arrow[d] & X \arrow[d] \\ 
S_i \arrow[r] & S 
\end{tikzcd} 
\]
are pullbacks for $i=1$ and $i=2$.  The result then 
follows from Lemma~\ref{lem:joyals descent lemma}, 
using the fact that the canonical 
map $S_1\sqcup S_2 \to S$ 
is surjective.  
\end{proof} 

\begin{remark} 
\label{rem:pushout descent}
If follows from the proof of Lemma~\ref{lem:pushout descent} 
that the canonical map 
\[
X_1\cup_{X_0}X_2\to S_1\cup_{S_0}S_2  
\]
is an inner fibration.  Note also that if 
$S_0\to S_1$ and $X_0\to X_1$ are inner anodyne 
then the canonical maps $S_2\to S$ and 
$X_2\to X$ are also inner anodyne.  
\end{remark} 

\subsection{Descent and the small object argument} 
\label{sec:descent and so arg} 
Suppose $\Sigma$ is a set of inner anodyne  
in $\sSet$ with small domains.  Let $S$ be a simplicial set.  
Construct an inner anodyne map $S\to T$ by the following 
mild variation of the small object argument.  
  Define a 
sequence of inner anodyne maps 
\[
S = S(0) \to S(1) \to S(2) \to \cdots 
\]
inductively as follows.  Assuming that $S(m)$ has been 
defined, suppose given a set $A_m$ of maps 
$\alpha\colon I_{\alpha} \to S(m)$, where $I_{\alpha}$ 
is the domain of a map in $\Sigma$ for each $\alpha\in A_m$.  
Define $S(m+1)$ by the pushout diagram 
\[
\begin{tikzcd} 
\bigsqcup_{\alpha\in A_m} I_{\alpha} 
\arrow[d,"{\sqcup f_{\alpha}}"] 
\arrow[r,"\phi"] & S(m) \arrow[d]           \\ 
\bigsqcup_{\alpha\in A_m} J_{\alpha} 
\arrow[r] & S(m+1) 
\end{tikzcd} 
\]
where $\phi$ is equal to the map $\alpha\colon I_{\alpha}
\to S(m)$ on the summand labelled by $\alpha$, and where 
$f_{\alpha}\colon I_{\alpha}\to J_{\alpha}$ 
is a map in $\Sigma$ for all $\alpha\in A_m$.  Let 
$T = \varinjlim S(m)$ so that we have a canonical map 
$S\to T$.  

\begin{proposition} 
\label{prop:tech prop1}
Suppose that every map $\alpha\colon I_{\alpha}\to S(m)$ 
in $A_m$ has the following property: if $Z\to S(m)$ is an inner 
fibration, then the induced inner fibration $\alpha^*Z\to I_{\alpha}$ 
satisfies inner anodyne descent with respect to the map $f_{\alpha}\colon 
I_{\alpha}\to J_{\alpha}$.  Then the following is true: 
every inner fibration $p\colon X\to S$ satisfies 
inner anodyne descent with respect to the map $S\to T$.  
\end{proposition}

\begin{proof} 
Define $X(0) = X$.  Under the above hypothesis it 
follows from Lemma~\ref{lem:pushout descent} and 
Remark~\ref{rem:pushout descent}
that we can construct inductively a sequence 
of inner fibrations $p(m)\colon X(m)\to S(m)$ such that 
for every $m\geq 0$ we have a pullback diagram 
\[
\begin{tikzcd} 
X(m) \arrow[d,"{p(m)}"] \arrow[r] & 
X(m+1) \arrow[d,"{p(m+1)}"]         \\ 
S(m) \arrow[r] & S(m+1) 
\end{tikzcd} 
\]
in which the horizontal maps are inner anodyne.  
Define $Y = \varinjlim X(m)$.  Then the canonical 
map $q\colon Y\to T$ is an inner fibration and the 
diagram 
\[
\begin{tikzcd} 
X \arrow[r] \arrow[d,"p"] 
& Y \arrow[d,"q"]        \\ 
S \arrow[r] & T 
\end{tikzcd} 
\]
is a pullback by Lemma~\ref{lem:seq colimit descent}.  
It is straightforward to check that the horizontal maps 
in this diagram are inner anodyne.    
\end{proof}

\subsection{Descent for inner anodyne maps} 
Every inner fibration $X\to \Lambda^2_1$ 
satisfies inner anodyne descent with respect to the inclusion 
$\Lambda^2_1\subseteq \Delta^2$.  
More precisely we have the following lemma.

\begin{lemma} 
\label{lem:descent for spine inclusions}
Suppose that $p\colon X\to \Lambda^2_1$ is an inner 
fibration of simplicial sets where $n\geq 2$.  
Then there is a pullback diagram 
\[
\begin{tikzcd} 
X\arrow[r,"v"]\arrow[d,"p"] 
& \calC \arrow[d,"q"] \\ 
\Lambda^2_1 \arrow[r] & \Delta^2 
\end{tikzcd} 
\]
in which the map $v$ is inner anodyne and $q$ is an inner fibration.    
\end{lemma} 

\begin{proof} 
We use a modification of the small object argument: we consider 
diagrams 
\begin{equation}
\tag{$D$}
\begin{tikzcd} 
\Lambda^m_i \arrow[d] 
\arrow[r] & X \arrow[d] \\
\Delta^m \arrow[r,"{\beta}"] & \Delta^2 
\end{tikzcd}
\end{equation}
in which $\beta$ does not factor 
through $\Lambda^2_1$ and where $0<m<i$.  
We define a simplicial set $\calC(1)$ by the pushout 
\[
\begin{tikzcd} 
\bigsqcup_{D}\Lambda^m_i \arrow[d] 
\arrow[r] & X  \arrow[d] \arrow[d] \\ 
\bigsqcup_{D} \Delta^m 
\arrow[r,"{\beta_D}"] & \calC(1) 
\end{tikzcd}
\]
in which the coproducts are taken over all 
diagrams $(D)$ as above.  The key point is that 
the canonical map $X\to \Lambda^2_1\times_{\Delta^2} 
\calC(1)$ is an isomorphism.  To see this it suffices 
to prove that for each diagram $(D)$ as above, 
the map $\beta$ induces an isomorphism 
\begin{equation}
\label{eq:p simplices}
(\Lambda^2_1\times_{\Delta^2}\Lambda^m_i)_{p} = 
(\Lambda^2_1\times_{\Delta^2}\Delta^m)_{p}
\end{equation}
on $p$-simplices for all $p\geq 0$.  If $p\leq m-2$ 
this is clear, since $(\Lambda^m_i)_p = 
(\Delta^m)_p$ in this case.  
Suppose that $\beta(\partial_i\Delta^{m-1})$ 
is the image of an $(m-1)$-simplex of 
$\Lambda^2_1$ in $\Delta^2$.  
This can only happen if $\beta$ restricts to a 
map $\partial_i\Delta^{m-1}\to \Delta^{\set{k,k+1}}$ for some $k\geq 0$.  
In particular $\beta(0),\beta(m)\in \set{k,k+1}$.  
Since $\beta(0)\leq \beta(i)\leq \beta(m)$ it follows 
that $\beta$ factors through $\Lambda^2_1$, 
contradicting the hypothesis that it does not.  
It follows that~\eqref{eq:p simplices}  
holds when $p=m-1$.  In a similar fashion one 
proves that~\eqref{eq:p simplices} holds for all $p\geq m$.  

Returning to the modification of the small object argument, we 
continue in this manner to obtain the usual sequence 
\[
X \to \calC(1)\to \calC(2) \to \cdots \to \calC(n)\to \cdots 
\]
where each map is inner anodyne.  The 
argument above shows that we have isomorphisms 
$X\to \Lambda^2_1\times_{\Delta^2} \calC(n)$ for 
each $n\geq 1$.  Let $\calC$ denote the colimit of the $\calC(n)$'s.  
The canonical map $X\to \calC$ is inner anodyne and induces an isomorphism 
$X\to \Lambda^2_1\times_{\Delta^2}\calC$.  The canonical 
map $\calC\to \Delta^2$ from the colimit is still an inner fibration, for given 
a commutative diagram 
\[
\begin{tikzcd} 
\Lambda^m_i \arrow[d] 
\arrow[r,"{\alpha}"] & \calC\arrow[d] \\ 
\Delta^m \arrow[r,"{\beta}"] & \Delta^2 
\end{tikzcd} 
\]
in which $0<i<m$, if the map $\beta$ does not factor through $\Lambda^2_1$ then 
a diagonal filler exists by construction.  
Otherwise, if $\beta$ does factor through $\Lambda^2_1$ 
then it must factor through some $\Delta^{\set{i,i+1}}$.  In this 
case the map $\alpha$ must factor through 
$X|\Delta^{\set{i,i+1}}$ and hence the required diagonal filler exists 
since $X|\Delta^{\set{i,i+1}}$ is an $\infty$-category.     
\end{proof} 

\begin{remark} 
In fact, it is easy to extend the above argument 
to show that every inner fibration $X\to I_n$ 
satisfies descent with respect to the spine 
inclusion $I_n\subseteq \Delta^n$ (see 
Example~\ref{ex:spine inclusion}). 
\end{remark}

The story for inner horn inclusions 
$\Lambda^n_i\subseteq \Delta^n$ 
with $n\geq 3$ is more complicated.  
In general it is not true that 
every inner fibration $X\to \Lambda^n_i$ 
satisfies descent with respect to the inner horn inclusion 
$\Lambda^n_i\subseteq \Delta^n$.  

\begin{example} 
The inner fibration 
$\Delta^{\set{0,2}}\cup \Delta^{\set{2,3}}
\to \Lambda^3_1$ does not satisfy descent 
with respect to the inner horn inclusion 
$\Lambda^3_1\subseteq \Delta^3$.  
\end{example} 


The following proposition gives a 
sufficient condition for an inner 
fibration over $\Lambda^n_i$ to satisfy 
descent with respect to the inclusion 
$\Lambda^n_i\subseteq \Delta^n$.  It is a 
generalization of Lemma 3.1.2.4 
from \cite{HA}.

\begin{proposition} 
\label{prop:key tech prop2}
Let $p\colon \calC\to \Lambda^n_i$ be an inner fibration where $n\geq 3$.  
Let $q$ denote the projection 
\[
q\colon \Delta^{\set{1,\ldots,n}}\times_{\Lambda^n_i} \calC\to 
\Delta^{\set{1,\ldots,n}}. 
\] 
Suppose that the following conditions are satisfied:  
\begin{enumerate}[label=(\roman*)]
\item for every vertex $x\in \calC_j$, $1\leq j\leq n-2$, there 
exists a $q$-cocartesian morphism $f\colon x\to y$ where 
$y\in \calC_{j+1}$; 

\item for every vertex 
$y\in \calC_j$, $2\leq j\leq n-1$ there exists a 
$q$-cartesian morphism $f\colon x\to y$ where 
$x\in \calC_{j-1}$. 
\end{enumerate}
Then $p\colon \calC\to \Lambda^n_i$ satisfies 
inner anodyne descent with respect to the inclusion 
$\Lambda^n_i\subseteq \Delta^n$.  
\end{proposition}

The proof of Proposition~\ref{prop:key tech prop2} 
is lengthy and is post-poned until 
Section~\ref{sec:proof of gen lurie lemma}.  
We can now state and prove the main result 
of this section.

\begin{proposition}
\label{prop:descent pre-fibrant simp sets}
Let $S$ be a simplicial set.  Then there 
exists an inner anodyne map 
$S\to T$, where $T$ is a pre-fibrant simplicial 
set, with the property that 
every inner fibration $X\to S$ 
satisfies inner anodyne descent with respect to 
$S\to T$.  
\end{proposition} 

\begin{proof} 
We use the modification of the small object argument 
described in Section~\ref{sec:descent and so arg} in 
which $\Sigma$ is the set of inner horn inclusions 
$\Lambda^n_i\subseteq \Delta^n$, $0<i<n$, in $\sSet$, 
and where $A_m$ is the set of maps $\alpha\colon 
\Lambda^n_i\to S(m)$ which satisfy $d_0(\alpha)$ 
is constant if $n>2$.  This produces an 
inner anodyne map $S\to T$, where 
$T$ is a pre-fibrant simplicial set.  
By Proposition~\ref{prop:tech prop1} 
it suffices to prove that for any map $\alpha\in A_m$, 
and for any inner fibration $Z\to S(m)$, the induced 
inner fibration $\alpha^*Z\to \Lambda^n_i$ satisfies 
inner anodyne descent with respect to the inclusion $\Lambda^n_i 
\subseteq \Delta^n$.  If $n=2$ this 
follows from Lemma~\ref{lem:descent for spine inclusions}.  
Therefore it suffices to 
verify that for any map $\alpha\colon \Lambda^n_i\to S(m)$ 
in $A_m$ with $n\geq 3$, and 
for any inner fibration $Z\to S(m)$, the 
hypotheses of Proposition~\ref{prop:key tech prop2} 
are satisfied by the inner fibration $\alpha^*Z\to 
\Lambda^n_i$.  Suppose then that $\alpha\colon \Lambda^n_i\to 
S(m)$ satisfies $d_0(\alpha)$ is constant.  Let $\calC
\to \Lambda^n_i$ denote the induced inner fibration 
$\alpha^*Z\to \Lambda^n_i$.  The projection 
\[
q\colon \Delta^{\set{1,\ldots,n}}\times_{\Lambda^n_i} 
\calC\to \Delta^{\set{1,\ldots,n}} 
\]
is the induced inner fibration 
\[
d_0(\alpha)^*Z\to \Delta^{\set{1,\ldots,n}} 
\]
Since $d_0(\alpha)$ is constant this induced inner 
fibration is isomorphic to one of the form 
\[
\Delta^{\set{1,\ldots,n}}\times Z_{v}\to 
\Delta^{\set{1,\ldots,n}} 
\]
where $Z_{v}$ denotes the fiber of $Z$ 
over some vertex $v\in S(m)$;  this  
inner fibration clearly satisfies the hypotheses 
of Proposition~\ref{prop:key tech prop2}
\end{proof}

\section{Inner anodyne maps and the inner model structure} 
\label{sec:inner anodyne stuff}

In this section we study inner anodyne maps in a 
little more detail.  We give some characterizations 
of the class of inner anodyne maps and prove 
Theorem~\ref{thm:C}.  We introduce the 
inner model structure and prove Theorem~\ref{thm:D}.  

\subsection{Examples and basic properties} 
Recall that the class of inner anodyne 
maps in $\sSet$ is the saturated class 
of monomorphisms generated by the 
inner horn inclusions $\Lambda^n_i\subseteq 
\Delta^n$, $0<i<n$.  
It follows immediately that 
every inner anodyne map is bijective on 
vertices, since the inner horn inclusions 
have this property.    

\begin{example} 
\label{ex:spine inclusion}
In addition to being bijective on 
vertices, every 
inner anodyne map is both left and right 
anodyne.  These three properties do not suffice to 
characterize inner anodyne maps, 
as the following example shows.
Regard $\Delta^1\times \Delta^1$ as a 
subcomplex of $\Delta^3$ by identifying 
$\Delta^1\times \Delta^1$ with 
$\Delta^{\set{0,1,3}}\cup 
\Delta^{\set{0,2,3}}$.  Then the inclusion 
$\Delta^1\times \Delta^1\subseteq 
\Delta^3$ is left and right anodyne, 
and is also bijective on vertices.  
But it is not inner anodyne.   
\end{example} 

An important example of inner anodyne 
maps are the {\em spine inclusions}. 

\begin{example} 
For any $n\geq 2$ the 
inclusion $I_n\subseteq \Delta^n$ 
is inner anodyne, where 
\[
I_n = \Delta^{\set{0,1}}\cup 
\cdots \cup \Delta^{\set{n-1,n}} 
\]
denotes the {\em spine} of 
$\Delta^n$ (for a proof that 
$I_n\subseteq \Delta^n$ is inner 
anodyne, see Proposition 2.13 of 
\cite{J2}).    
\end{example} 

\begin{definition}
Recall that a class $\calA$ of monomorphisms 
in $\sSet$ is said to satisfy the {\em right 
cancellation property} if the following condition 
is satisfied: if $u\colon A\to B$ and $v\colon B\to C$ 
are monomorphisms in $\sSet$ such that 
$vu, u\in \calA$, then $v\in \calA$.  
\end{definition} 

It is well-known that the classes of left 
anodyne and right anodyne maps in $\sSet$ 
have the right cancellation property.  
The same is also true for the class 
of inner anodyne maps.    

\begin{lemma}[\cite{S2}]
\label{lem:right cancellation inner} 
The class of inner anodyne morphisms in 
$\sSet$ satisfies the right cancellation 
property.  
\end{lemma}  

We shall make extensive use of 
Lemma~\ref{lem:right cancellation inner} 
in Section~\ref{sec:proof of gen lurie lemma} 
but for now let us note that it 
quickly gives the following characterization 
of the class of inner anodyne maps in 
$\sSet$.  

\begin{proposition} 
The class of inner anodyne maps 
in $\sSet$ is the smallest 
saturated class of monomorphisms 
in $\sSet$ which contains the spine inclusions 
$I_n\subseteq \Delta^n$ for 
every $n\geq 2$ and which 
has the right cancellation 
property.  
\end{proposition} 

\begin{proof} 
Let $\calA$ be a saturated 
class of monomorphisms which contains 
the spine inclusions $I_n\subseteq 
\Delta^n$ for all $n\geq 2$ and 
which has the right cancellation 
property.  Then $\calA$ contains 
the class of inner anodyne maps 
by Lemma 3.5 of \cite{JT}.  
Conversely, the class of inner 
anodyne maps is saturated, has the 
right cancellation property, and 
contains the spine inclusions.  
The result follows.    
\end{proof} 

One should beware that the 
weakly saturated class 
generated by the set of spine inclusions 
$I_n\subseteq \Delta^n$, $n\geq 2$, 
is not equal to the class of inner anodyne 
maps, as the 
following example shows.  

\begin{example} 
Let $S$ be the 
simplicial set defined by the pushout 
diagram 
\[
\begin{tikzcd} 
I_3 \arrow[d,hook] \arrow[r,hook] 
& \partial\Delta^3 \arrow[d] \\
\Delta^3 \arrow[r] & S 
\end{tikzcd} 
\]
Then for every $n\geq 2$, any 
map $I_n\to S$ extends along the 
inclusion $I_n\subseteq \Delta^n$.  
But the composite map $\Lambda^3_1\to 
\partial\Delta^3\to S$ does 
not extend along the inclusion 
$\Lambda^3_1\subseteq \Delta^3$.  
\end{example}

\subsection{Further properties of inner anodyne maps} 
Our aim in this section is to prove 
Theorem~\ref{thm:C} from the introduction.  
We begin by proving a special case of it, 
Lemma~\ref{lem:easy version of Joyals conj} below.   
We then give some different, equivalent reformulations of Theorem~\ref{thm:C}.  
We prove one of these reformulations, 
Theorem~\ref{thm:main thm for inner fibs} below.    

The following lemma gives a characterization 
of the inner anodyne maps with codomain an 
$\infty$-category $\calC$.  

\begin{lemma} 
\label{lem:easy version of Joyals conj} 
Suppose that $u\colon A\to \calC$ is a monic 
categorical equivalence which is a bijection 
on objects.  If $\calC$ is an $\infty$-category 
then $u$ is inner anodyne.  
\end{lemma} 

\begin{proof} 
Factor $u$ as $u = pi$ where $i\colon A\to \calB$ is inner anodyne 
and $p\colon \calB\to \calC$ is an inner fibration.  
Since the functors $\h(u)$ and $\h(i)$ are isomorphisms 
it follows that $\h(p)$ is an isomorphism.  Since 
$\calC$ is an $\infty$-category it follows quickly that 
$p$ is a categorical fibration.    
Since $p$ is a categorical 
equivalence between $\infty$-categories it follows that 
$p$ is a trivial Kan fibration 
(Lemma~\ref{lem:ps fibn + ce implies tkf} 
and Proposition~\ref{prop:ff + ess surj iff equiv}).  
Choosing a section of $p$ exhibits $u$ as a codomain retract 
of the inner anodyne map $i$.  Hence $u$ is inner anodyne.    
\end{proof} 

\begin{example} 
In particular, the canonical inclusion 
$\Delta^n\times \set{0}\cup \partial\Delta^n 
\times J \subseteq \Delta^n\times J$ is inner 
anodyne for every $n\geq 1$.  
\end{example} 

This next proposition lists some reformulations 
of Theorem~\ref{thm:C}.  We shall see below 
that the third statement in this list 
can be proven by a descent argument 
using pre-fibrant simplicial sets. 

\begin{proposition}
\label{prop:equiv statements for ia}
The following statements are equivalent:
\begin{enumerate} 
\item The class of inner anodyne maps satisfies the 
2-out-of-3 property in the class of monomorphisms 
in $\sSet$; 
\item Every monic categorical equivalence which 
is a bijection on objects is inner anodyne; 
\item if $p\colon X\to S$ is an inner fibration 
between simplicial sets which is a bijection on 
objects and a categorical equivalence, then 
$p$ is a trivial Kan fibration.  
\end{enumerate} 
\end{proposition} 

We clarify the meaning of statement (1) above: if $u$ 
and $v$ are composable monomorphisms in $\sSet$ such that 
any two of the maps $u$, $v$ and $vu$ are inner anodyne, 
then all three of the maps are inner anodyne.  

\begin{proof} 
We prove that (1) $\implies$ (2).  Suppose that 
the class of inner anodyne maps satisfies the 
2-out-of-3 property.  Let $u\colon A\to B$ be 
a monic categorical equivalence which is a bijection 
on objects.  Choose an inner anodyne map 
$B\to B'$ where $B'$ is an $\infty$-category.  
Then the composite map $A\to B\to B'$ is a 
monic categorical equivalence and a bijection 
on objects with codomain an $\infty$-category.  
Hence it is inner anodyne (Lemma~\ref{lem:easy version of Joyals conj}).  
Therefore, under the assumption (1), it follows 
that $u\colon A\to B$ is inner anodyne.  

We prove that (2) $\implies$ (1).  Suppose that 
(2) is true.  Let $u\colon A\to B$ 
and $v\colon B\to C$ be monomorphisms 
of simplicial sets such that two of 
the maps $u$, $v$ and $vu$ are inner 
anodyne.  Under this assumption the 
2-out-of-3 property of categorical equivalences 
shows that all of the maps $u$, $v$ and 
$vu$ are categorical equivalences.  Similarly 
all of the maps $u$, $v$ and $vu$ must be 
bijections on objects.  It follows, under 
the assumption of (2), that all of the maps 
$u$, $v$ and $vu$ are inner anodyne.  

We prove that (2) $\implies$ (3).  Let 
$p\colon X\to S$ be an inner fibration between 
simplicial sets which is a categorical equivalence 
and a bijection on objects.  Choose an 
inner anodyne map $S\to S'$ where $S'$ is an 
$\infty$-category.  We may factor the composite 
map $X\to S\to S'$ as $X\to X'\to S'$ where 
$X\to X'$ is inner anodyne and $p'\colon X'\to S'$ 
is an inner fibration, so that we have a 
commutative diagram 
\[
\begin{tikzcd} 
X\arrow[d,"p"] \arrow[r] & X'\arrow[d,"p'"] \\ 
S\arrow[r] & S' 
\end{tikzcd} 
\]
The induced map $\h(X')\to \h(S')$ is an isomorphism 
of categories, from which it follows easily that 
$p'\colon X'\to S'$ is a categorical fibration.    
Since $p'$ is a categorical 
equivalence between 
$\infty$-categories it follows
(Lemma~\ref{lem:ps fibn + ce implies tkf}) 
that $p'$ is a trivial Kan fibration.  
The induced map $X\to X'\times_{S'}S$ is a monic 
categorical equivalence which is a bijection on objects.  
Hence it is inner anodyne, under the assumption that 
(2) holds.  Therefore, using the fact that $p$ is 
an inner fibration, we can find the indicated 
diagonal filler in the diagram 
\[
\begin{tikzcd} 
X \arrow[r,"\mathrm{id}_X"] \arrow[d] & X \arrow[d,"p"] \\ 
X'\times_{S'}S \arrow[r] \arrow[ur,dashed] & S 
\end{tikzcd} 
\]
It follows that $p\colon X\to S$ is a retract of 
the trivial Kan fibration $X\times_{S'}S\to S$, and 
hence must itself be a trivial Kan fibration.  

Finally, to complete the proof it suffices 
to prove that (3) $\implies$ (2).  Under the assumption 
of (3), the proof of Lemma~\ref{lem:easy version of Joyals conj} 
can be easily adapted to show that (2) is true.   
\end{proof} 

\begin{theorem} 
\label{thm:main thm for inner fibs}
Let $p\colon X\to S$ be an inner fibration between simplicial sets.  
If $p$ is a categorical equivalence and if the map $p_0\colon 
X_0\to S_0$ on objects is a bijection then $p$ is 
a trivial Kan fibration.  
\end{theorem} 

\begin{proof} 
By Proposition~\ref{prop:descent pre-fibrant simp sets} 
we may find a pullback diagram 
\[
\begin{tikzcd} 
X \arrow[r] \arrow[d,"p"] & X' \arrow[d,"p'"] \\ 
S \arrow[r] & S' 
\end{tikzcd} 
\]
in which the horizontal maps are inner anodyne, 
$p'\colon X'\to S'$ is an inner fibration and 
$S'$ (and hence $X'$) is a pre-fibrant simplicial set.  


Clearly $p'\colon X'\to S'$ is a categorical equivalence 
and also a bijection on objects.  It follows that without 
loss of generality we may suppose at the outset that 
$S$ (and hence $X$ --- see Remark~\ref{rem:condition *}) 
is a pre-fibrant simplicial set.  

Our aim is to prove that $p\colon X\to S$ has the 
right lifting property against the boundary inclusions 
$\partial\Delta^n\subseteq \Delta^n$ for $n\geq 0$.  
The case $n=0$ is clear, since $p$ is a bijection on 
objects.  

Suppose that $n\geq 1$ and suppose given 
a commutative diagram 
\begin{equation}
\label{eq:boundary horn filling diagram}
\begin{tikzcd} 
\partial\Delta^n \arrow[d] \arrow[r,"u"] & X \arrow[d,"p"] \\ 
\Delta^n \arrow[r] & S 
\end{tikzcd} 
\end{equation}
When $n\geq 2$ we may write 
\[
\partial\Delta^n = \Delta^{\set{0}}\star 
\partial\Delta^{\set{1,2,\ldots, n}}\cup 
\Delta^{\set{1,2,\ldots, n}}, 
\]
from which it follows that a diagonal filler 
exists in~\eqref{eq:boundary horn filling diagram} 
if and only if the 
indicated diagonal filler exists in the diagram 
\[
\begin{tikzcd} 
\partial\Delta^{\set{1,2,\ldots,n}} \arrow[d] 
\arrow[r] & X_{x/} \arrow[d]                    \\ 
\Delta^{\set{1,2,\ldots,n}} \arrow[r] 
\arrow[ur,dashed] & 
S_{x/}\times_S X 
\end{tikzcd} 
\]
where $x = u(0)$ and where the right hand 
vertical map is the induced map.  Since $p\colon X\to S$ is 
an inner fibration the induced map $X_{x/}\to 
S_{x/}\times_S X$ is a left fibration.  Therefore 
it suffices to prove that the induced map has 
contractible fibers (Lemma 2.1.3.4 of 
\cite{HTT}).  Let $y$ be a vertex of 
$X$.  It suffices to prove that the left fibration  
\[
X_{x/}\times_X \set{y}\to S_{x/}\times_S \set{y} 
\]
has contractible fibers.  But this latter map 
is a trivial Kan fibration by the hypothesis that 
$p\colon X\to S$ is a categorical equivalence,  
using Proposition~\ref{prop:key prop} and the 
fact that $X$ and $S$ are pre-fibrant simplicial sets 
(see Remark~\ref{rem:fully faithful for pre-fibrant simp set}).. 

Finally, we need to prove that there is a 
diagonal filler 
in~\eqref{eq:boundary horn filling diagram} 
in the 
case when $n=1$.  In this case, write 
$x = u(0)$ again and write $y= u(1)$.  Then 
the canonical map 
\[
X_{x/}\times_X \set{y}\to S_{x/}\times_S \set{y} 
\]
is a trivial Kan fibration and in particular 
is surjective on vertices.  It follows that 
the edge $\Delta^1\to S$ can be lifted to an 
edge $\Delta^1\to X$, as required. 
\end{proof} 

We can now give the proof of Theorem~\ref{thm:C} 
from the introduction.  

\begin{proof}[Proof of Theorem~\ref{thm:C}]
We have observed earlier that every inner anodyne 
map is a bijection on objects and a categorical 
equivalence.  The converse statement follows 
from Theorem~\ref{thm:main thm for inner fibs} 
and Proposition~\ref{prop:equiv statements for ia}. 
\end{proof}

\subsection{The inner model structure} 

Let $\sSet(O)$ denote the subcategory of 
$\sSet$ whose objects consist of the 
simplicial sets whose set of $0$-simplices 
is equal to a fixed set $O$.  The morphisms 
in $\sSet(O)$ are the simplicial maps 
which induce the identity on the set $O$.  

Since limits and colimits are computed pointwise 
in $\sSet$, it follows that the inclusion 
$\sSet(O)\subseteq \sSet$ creates all 
limits and colimits, and hence that 
$\sSet(O)$ is complete and cocomplete.

We will say that a map $X\to Y$ in 
$\sSet(O)$ is an inner fibration, or 
inner anodyne, or a categorical 
equivalence, or a trivial 
Kan fibration, just in case the underlying 
map of simplicial sets is such a map.  

We make the following 
observation.  

\begin{lemma} 
A map $u\colon A\to B$ in 
$\sSet(O)$ is a monomorphism 
if and only it has the left 
lifting property with respect 
to all trivial Kan fibrations 
in $\sSet(O)$.  
\end{lemma} 

\begin{proof} 
Observe that $u\colon A\to B$ in 
$\sSet(O)$ is a monomorphism 
if and only if its image in 
$\sSet$ is a monomorphism, if 
and only if its image in 
$\sSet$ has the left lifting 
property with respect to all 
trivial fibrations in $\sSet$.  
Therefore, if $u\colon A\to B$ is a 
monomorphism in $\sSet(O)$ then 
it has the left lifting property 
with respect to all trivial Kan fibrations 
in $\sSet(O)$.  We need to prove 
that the converse statement is 
true.  

Suppose given a commutative 
diagram 
\begin{equation} 
\label{eq:comm diag2}
\begin{tikzcd} 
A \arrow[d,"u"] \arrow[r] 
& X \arrow[d,"p"] \\ 
B \arrow[r] 
\arrow[ur,dashed] & S 
\end{tikzcd} 
\end{equation}
where $p\colon X\to S$ is a 
trivial Kan fibration in 
$\sSet$ and where $u\colon A\to B$ 
is a monomorphism 
in $\sSet(O)$.  
Define an object $S_O$ of 
$\sSet(O)$ by the pullback diagram 
\[
\begin{tikzcd} 
S_O \arrow[r] \arrow[d] 
& S \arrow[d]          \\ 
\cosk_0 O \arrow[r] 
& \cosk_0 S 
\end{tikzcd} 
\]
Define similarly an object 
$X_O$ of $\sSet(O)$ and observe 
that $p$ induces a map 
$p_O\colon X_O\to S_O$ 
which is a trivial Kan fibration.  We 
have a commutative diagram 
\begin{equation}
\label{eq:comp diag1}
\begin{tikzcd} 
A \arrow[d,"u"] 
\arrow[r] & X_O 
\arrow[d,"p_O"] 
\arrow[r] & X 
\arrow[d,"p"]         \\ 
B \arrow[ur,dashed] 
\arrow[r] & 
S_O \arrow[r] & S 
\end{tikzcd} 
\end{equation}
A choice of the indicated diagonal 
filler in~\eqref{eq:comp diag1}
induces a 
diagonal filler 
in~\eqref{eq:comm diag2}.  
The result follows. 
\end{proof} 

In particular it follows 
from the lemma that 
a map 
in $\sSet(O)$ is a trivial 
Kan fibration if and only 
if it has the right lifting 
property against the set of 
maps in $\sSet(O)$ of the 
form 
\[
O\cup_{\set{0,\ldots,n}}\partial\Delta^n 
\to O\cup_{\set{0,\ldots, n}} 
\Delta^n 
\]
for some map $\set{0,\ldots, n}
\to O$, where $n\geq 1$.  

A similar argument 
shows that the class of inner anodyne maps  
in $\sSet(O)$ is precisely the class of 
maps which belong to the weakly 
saturated class in $\sSet(O)$ 
generated by  
the maps in $\sSet(O)$ of the form 
\[
O\cup_{\set{0,\ldots, n}}\Lambda^n_i 
\to O\cup_{\set{0,\ldots,n}}\Delta^n 
\]
for some map $\set{0,\ldots, n}\to O$, 
where $0<i<n$.  

We can now establish the inner 
model structure on $\sSet(O)$, proving 
Theorem~\ref{thm:D} from Section~\ref{sec:intro}.  
We have the following theorem. 

\begin{theorem} 
\label{thm:inner model structure}
There is the structure of a model category 
on $\sSet(O)$ in which 
\begin{itemize} 
\item the cofibrations are the monomorphisms 
in $\sSet(O)$; 
\item the weak equivalences are the 
categorical equivalences in $\sSet(O)$; and 
\item the fibrations are the inner fibrations 
in $\sSet(O)$.  
\end{itemize} 
The model structure is cofibrantly generated 
and left proper.  
\end{theorem} 


\begin{proof} 
Let $\calW$ denote the class of categorical 
equivalences in $\sSet(O)$.   Clearly 
$\calW$ satisfies the 2-out-of-3 property.  
Let $\calC$ denote the class of monomorphisms 
in $\sSet(O)$ and let $\calF_{\calW}$ 
denote the class of trivial Kan fibrations 
in $\sSet(O)$.  Note that 
$\calF_{\calW}\subseteq 
\calF\cap \calW$.
By the small object argument, 
the pair $(\calC,\calF_{\calW})$ forms a 
weak factorization system on $\sSet(O)$.  

Let $\calF$ denote the class of inner fibrations 
in $\sSet(O)$, and let $\calC_{\calW}$ denote 
the class of inner anodyne maps in 
$\sSet(O)$.  Note that $\calC_{\calW} 
\subseteq \calC\cap \calW$.  
Again, by the small object argument 
the pair $(\calC_{\calW},\calF)$ forms a 
weak factorization system on $\sSet(O)$.  
It suffices (see Proposition E.1.11 of 
\cite{J2}) to prove the 
reverse inclusion $\calF\cap \calW\subseteq 
\calF_{\calW}$.  This follows immediately from 
Theorem~\ref{thm:main thm for inner fibs}.  
\end{proof} 

\section{The Joyal model structure} 
\label{sec:joyal model structure} 

In this section we prove Theorem~\ref{thm:A} 
and Theorem~\ref{thm:B} and prove 
the existence of the Joyal 
model structure (see Theorem~\ref{thm:joyal nodel str} 
below).

\subsection{Generating acyclic cofibrations} 
In this section we will define a set of 
generating acyclic cofibrations 
for the Joyal model structure on $\sSet$.  
This set will turn out to be analogous 
to the set of generating ayclic cofibrations 
for the Bergner model structure on $\SCat$.  

\begin{notation} 
Denote by $\calJ$ a chosen set 
of representatives for the isomorphism classes of 
simplicial sets $B$ with two vertices $0$ 
and $1$, with countably many non-degenerate 
simplices and 
such that the inclusion $\set{0}\hookrightarrow 
B$ is a categorical equivalences.  
\end{notation} 

In the next section we will prove that a set of generating acyclic cofibrations 
for the Joyal model structure is given by the set 
$\calA$, which is equal to the union of the set $\calJ$ 
and the set of inner anodyne inclusions $\Lambda^n_i\subseteq 
\Delta^n$, $0<i<n$. 



In this section, we characterize the maps in $\sSet$ 
which have the right lifting property with respect 
to the maps in $\calA$.  More precisely, we will 
prove the following proposition.  

\begin{proposition} 
\label{prop:lifting wrt A1 and A2}
A map $p\colon X\to S$ in $\sSet$ has the right lifting property 
with respect to the maps in the set $\calA$ if and only 
if it is an inner fibration such that the induced 
functor $\h(p)\colon \h(X)\to \h(S)$ is an 
isofibration.  
\end{proposition} 

Proposition~\ref{prop:lifting wrt A1 and A2} 
will follow immediately from Propositions~\ref{prop:bergners 2.3} 
and~\ref{prop:bergners 2.5} below, which 
are the direct analogs for $\sSet$ of Proposition 2.3 
and Proposition 2.5 from \cite{Bergner}.  
We shall see that, with some work, the 
arguments used in \cite{Bergner} to 
establish these propositions can be adapted 
to the setting of simplicial sets.   

We shall need the following analog 
of Lemma 2.4 from \cite{Bergner}.  
It is closely 
related to the {\em bounded cofibration} 
property introduced in \cite{Jardine}.

\begin{lemma} 
\label{lem:bergners lemma}
Let $S$ be a pre-fibrant simplicial set with two vertices 
$x$ and $y$ and suppose that $f\colon x\to y$ 
is an equivalence in $S$.  Then the map 
$\Delta^1\to S$ classifying the edge $f$ factors as 
$\Delta^1\to B\to S$ in such a way 
that the composite $\set{0}\to B$ belongs 
to the set $\calJ$. 
\end{lemma}

The proof of Lemma~\ref{lem:bergners lemma} 
is somewhat complicated and its proof 
is deferred until Section~\ref{sec:proof of bergners lemma}.  

To begin with, we prove the following version of Proposition 2.3 
from \cite{Bergner}, adapting the proof of that  
proposition from op.\ cit.\ to the present context. 

\begin{proposition} 
\label{prop:bergners 2.3}
Suppose that $p\colon X\to S$ is an inner fibration 
of simplicial sets which has the right 
lifting property against the set of maps 
$\set{0}\to B$ in $\calJ$.  Then $\h(f)\colon 
\h(X)\to \h(S)$ is an isofibration.  
\end{proposition}

\begin{proof} 
By Proposition~\ref{prop:descent pre-fibrant simp sets} 
there exists a pullback 
diagram 
\[
\begin{tikzcd} 
X \arrow[d,"p"] \arrow[r,"i"] & Y \arrow[d,"q"] \\ 
S \arrow[r,"j"] & T 
\end{tikzcd} 
\]
in which the horizontal arrows are inner anodyne maps, 
$q$ is an inner fibration, and where $T$ 
is a pre-fibrant simplicial set.  
It follows that without loss of generality we may 
suppose that $S$ and $X$ are pre-fibrant simplicial sets.      

Suppose that $x\in X$ and that $f\colon s\to t$ is an 
equivalence in $S$, where $s = p(x)$.  We will show 
that there is an equivalence $u\colon x\to y$ in $X$ 
such that $p(u) = f$.  
Suppose to begin with that $s\neq t$.  Let $\Delta^1
\to S$ classify the edge $f$.  Write $S'$ for the full 
simplicial subset of $S$ spanned by $s$ and $t$.  
We claim that $f$ restricts to an 
equivalence in $S'$.  To see this, it 
suffices to prove that $\h(S')\subseteq \h(S)$ is a 
full subcategory.  To see this, observe that   
by Lemma~\ref{lem:pullback lemma}, the pullback 
diagram 
\[
\begin{tikzcd} 
S' \arrow[r] 
\arrow[d] & 
\cosk_0\set{s,t} \arrow[d] \\ 
 S
\arrow[r] & 
\cosk_0 S 
\end{tikzcd} 
\]
of pre-fibrant simplicial sets shows that the 
induced diagram 
\[
\begin{tikzcd} 
\h(S') \arrow[r] 
\arrow[d] & 
\h(\cosk_0\set{s,t}) \arrow[d] \\ 
\h(S)  \arrow[r] 
& \h(\cosk_0 S)
\end{tikzcd} 
\]
is also a pullback, which immediately 
implies the claim.    Using 
Lemma~\ref{lem:bergners lemma} we may 
factor the induced map $\Delta^1\to S'$ 
as $\Delta^1\to B\to S'$ in such a way that 
the map $\set{0}\to B$ belongs to $\calJ$.  

By hypothesis we can find the indicated diagonal filler 
in the diagram 
\[
\begin{tikzcd} 
\set{0} \arrow[d] \arrow[rr] & & X \arrow[d,"p"] \\ 
B \arrow[r] \arrow[urr,dashed] & S' \arrow[r] & S 
\end{tikzcd} 
\]
Pre-composing the map $B\to X$ with the map 
$\Delta^1\to B$ defines an edge $u\colon \Delta^1
\to X$ which lifts $f$ and is an equivalence.  

Suppose now that $s=t$. Let $C$ denote the category with 
two objects $s$ and $s'$ and in which 
every set of morphisms in $C$ is 
equal to the set $\h(S)(s,s)$.  Composition 
in $C$ is given by composition in $\h(S)$.  
There is a canonical functor $C\to \h(S)$ 
which maps every object to $s$.    
Define a simplicial set $S'$ by the 
pullback diagram 
\[
\begin{tikzcd} 
S' \arrow[r] \arrow[d] & S \arrow[d] \\ 
C \arrow[r] & \h(S) 
\end{tikzcd} 
\]
Since $S$ is pre-fibrant and $C\to \h(S)$ is an 
inner fibration it follows that $S'$ is 
pre-fibrant.  Moreover $\h(S') = C$ by 
Lemma~\ref{lem:pullback lemma}.  The edge 
$f\colon \Delta^1\to S$ induces an 
edge $\tilde{f}\colon \Delta^1\to S'$ which 
is clearly an equivalence (its 
image in $C$ is an isomorphism).   The argument 
then proceeds as in the previous case. 
\end{proof} 

Next, we prove the following version of 
Proposition 2.5 from \cite{Bergner}.  Again, 
the proof is an adaptation of the arguments 
used in \cite{Bergner} to the present context. 

\begin{proposition} 
\label{prop:bergners 2.5}
Suppose that $p\colon X\to S$ is an inner fibration 
between simplicial sets such that the induced functor 
$\h(X)\to \h(S)$ is an isofibration.  Then $p$ has 
the right lifting property with respect to the 
maps in the set $\calJ$. 
\end{proposition} 

\begin{proof} 
We need to show that the indicated 
diagonal filler exists in any diagram 
of the form 
\begin{equation}
\label{eq:pullback for A2}
\begin{tikzcd} 
\set{0} \arrow[d] \arrow[r] & X \arrow[d,"p"] \\ 
B   \arrow[r] \arrow[ur,dashed] &  S
\end{tikzcd} 
\end{equation}
where $\set{0}\to B$ is a map in $\calJ$.  
Proposition~\ref{prop:descent pre-fibrant simp sets} 
shows that without loss of generality 
we may suppose $S$, and hence $X$, 
are pre-fibrant simplicial sets.   

Using Remark~\ref{rem:condition *} 
and the small object argument, we see that 
we may choose an inner anodyne map $B\to B'$, 
where $B'$ is a pre-fibrant simplicial set, 
such that the map $B\to S$ factors 
through $B'$.  
Note that $\set{0}\to B'$ is a categorical 
equivalence and that $B'$ can be 
chosen to have  
countably many non-degenerate simplices.  
It follows that without loss of generality 
we may assume at the outset that $B$ is a 
pre-fibrant simplicial set.  

Let $g\colon 0\to 1$ be an equivalence 
in $B$ and write $g\colon \Delta^1\to B$ 
for the simplex classifying $g$.  Under 
the map $B\to S$, $g$ is mapped to 
an equivalence in $S$.  
Since the functor $\h(X)\to \h(S)$ 
is an isofibration, and $X$ and $S$ are 
pre-fibrant simplicial sets, we can use 
Lemma~\ref{lem:char of isofibns} to find the indicated 
map making the following diagram commute:  
\[
\begin{tikzcd} 
\set{0} \ar[d] \ar[r] & X \ar[dd,"p"] \\ 
\Delta^1 \ar[ur,dotted,"f"] \ar[d,"g"] & \\ 
B \ar[r] & S 
\end{tikzcd} 
\]
We need to show that we can extend $f\colon \Delta^1
\to X$ to a map $B\to X$ making the 
diagram~\eqref{eq:pullback for A2} commute.  Form a pullback 
diagram 
\[
\begin{tikzcd} 
Y \ar[r] \ar[d] & 
X \ar[d,"p"] \\ 
B \ar[r] & S 
\end{tikzcd} 
\]
and let $\tilde{f}\colon \Delta^1\to Y$ 
denote the map induced by $f\colon \Delta^1\to X$.  
It follows from Lemma~\ref{lem:pullback lemma} 
that $\tilde{f}$ is an equivalence in $Y$.  

Let $Z$ denote the full simplicial subset of 
$Y$ spanned by the objects $\tilde{f}(0)$ 
and $\tilde{f}(1)$; clearly $\tilde{f}(0)\neq \tilde{f}(1)$.  
We regard $Z$ as an object of $\sSet(\set{0,1})$   
Since $\tilde{f}$ is an equivalence in $Y$, 
and $Z\subseteq Y$ is a full simplicial 
subset, it follows that $\tilde{f}$ is an equivalence 
in $Z$.  

Observe that $Z\to B$ is an inner fibration in $\sSet(\set{0,1})$, 
since it is the composite of the inner fibrations 
$Z\to Y$ and $Y\to B$.  In particular 
it follows that $Z$ is a pre-fibrant simplicial set.  
Lemma~\ref{lem:bergners lemma} implies that the 
map $\tilde{f}\colon \Delta^1\to 
Z$ factorizes as 
$\Delta^1\to B'\xrightarrow{\alpha} Z$, in such a way that the 
map $\set{0}\to B'$ belongs to $\calJ$.  
Since both maps $\set{0}\to B$ and $\set{0}\to 
B'$ are categorical equivalences, a 2-out-of-3 argument 
shows that the composite map $B'\xrightarrow{\alpha} Z\to B$ is a 
categorical equivalence.  

We have a commutative diagram 
\begin{equation}
\label{eq:pullback for A3}
\begin{tikzcd} 
\Delta^1 \ar[r,equal] \ar[d] & \Delta^1 
\ar[r,"\tilde{f}"] \ar[d] & Z \ar[d] \\ 
B' \ar[r,"\sim"] & B \ar[dotted,ur] \ar[r,equal] & B 
\end{tikzcd} 
\end{equation}
in $\sSet(\set{0,1})$.  Lemma 2.6 from 
\cite{Bergner}, applied to the inner model structure on 
$\sSet(\set{0,1})$ (Theorem~\ref{thm:inner model structure}), 
shows that the indicated diagonal 
filler $\beta\colon B\to Z$ 
exists in the diagram~\eqref{eq:pullback for A3}.  Composing 
$\beta$ with the canonical map $Z\to Y$ defines a diagonal filler 
for the diagram~\eqref{eq:pullback for A2}.    
\end{proof}

\subsection{The Joyal model structure} 
In this section we give a proof 
of the existence of the Joyal model 
structure on $\sSet$.  As immediate outcomes 
of our proof we obtain Theorem~\ref{thm:A} 
and Theorem~\ref{thm:B}.  

Recall that the Joyal model structure 
is uniquely determined by the fact 
that its fibrant objects are the 
$\infty$-categories and its cofibrations 
are the monomorphisms (see \cite{DS}).  
Recall also that the fibrations in the 
Joyal model structure are called 
{\em categorical fibrations} in 
\cite{HTT}.

\begin{theorem}[Joyal/Lurie] 
\label{thm:joyal nodel str}
There is the structure of a model 
category on $\sSet$ for which 
\begin{itemize} 
\item the cofibrations are the 
monomorphisms, 
\item the fibrant objects are the 
$\infty$-categories, 
\item the weak equivalences 
are the categorical equivalences.  
\end{itemize} 
The model structure is cofibrantly generated 
and left proper.  
\end{theorem} 

\begin{proof} 
The class $\calW$ of categorical 
equivalences in $\sSet$  
satisfies the 2-out-of-3 property.  
Let $\calC_{\calW}$ denote the 
class of monomorphisms generated 
by the set of maps $\calA$ and 
let $\calF$ denote the class of maps 
which have the right lifting property 
against the maps in $\calA$.    
The small object argument gives 
a weak factorization system 
$(\calC_{\calW},\calF)$.  

Let 
$\calF_{\calW}$ denote the 
class of trivial Kan fibrations and 
let $\calC$ denote the class 
of monomorphisms in $\sSet$.  
The small object argument gives 
a weak factorization system 
$(\calC,\calF_{\calW})$.  
Clearly 
we have $\calC_{\calW}\subseteq 
\calC\cap \calW$ and $\calF_{\calW}\subseteq 
\calF\cap \calW$.  To 
prove that the classes $\calC$, 
$\calF$ and $\calW$ determine a model 
structure it suffices 
(see Proposition E.1.11 of \cite{J2})
to prove that $\calF\cap \calW\subseteq 
\calF_{\calW}$.  

Suppose that $p\colon X\to S$ belongs 
to the class $\calF\cap\calW$.  Then 
$p$ is an inner fibration such that 
$\h(p)\colon \h(X)\to \h(S)$ is an 
isofibration (Proposition~\ref{prop:lifting wrt A1 and A2}), 
and $p$ is a categorical 
equivalence.  There exists a commutative diagram 
\[
\begin{tikzcd} 
X \arrow[r,"u"] \arrow[d,"p"] & 
Y \arrow[d,"q"] \\ 
S \arrow[r,"v"] & T 
\end{tikzcd} 
\]
in which $u$ and $v$ are inner 
anodyne, $q\colon Y\to T$ is an 
inner fibration and $T$ is an 
$\infty$-category.  It follows 
that $q$ is a categorical fibration.  
Since $q$ is also a categorical 
equivalence it follows that 
$q$ is a trivial Kan fibration 
(Lemma~\ref{lem:ps fibn + ce implies tkf}).  
Therefore the horizontal map 
\[
\begin{tikzcd}[column sep = small] 
X \arrow[rr] \arrow[dr,"p"'] & 
& S\times_{T}Y \arrow[dl]      \\ 
& S 
\end{tikzcd} 
\]
in the induced diagram is a 
categorical equivalence.  
Since it is also a monomorphism 
which is bijective on objects 
it must be inner anodyne 
(Theorem~\ref{thm:C}).  
Therefore, since $p\colon X\to S$ 
is an inner fibration, $p$ is a 
retract of the trivial Kan fibration 
$S\times_{T}Y\to S$ and hence 
is itself a trivial Kan fibration.  

To complete the proof we need to 
show that the fibrant objects for 
this model structure are the $\infty$-categories.  
It is clear that every fibrant 
object is an $\infty$-category.  
We must prove that if $\calC$ is 
an $\infty$-category then 
the canonical map $\calC\to 
\Delta^0$ has the right lifting 
property with respect to every map 
$\set{0}\to B$ in the set $\calJ$.  
This is clear however, since we 
can just retract $B$ onto $\set{0}$.  
\end{proof}  
  
As immediate corollaries of our 
proof of Theorem~\ref{thm:joyal nodel str} 
we obtain Theorem~\ref{thm:A} 
and Theorem~\ref{thm:B}.

\begin{proof}[Proof of Theorem~\ref{thm:A}]
From the proof of Theorem~\ref{thm:joyal nodel str} 
we see that a map is a categorical fibration 
if and only if it has the right lifting property 
with respect to the maps belonging to 
the set $\calA$.  The result then follows from 
Proposition~\ref{prop:lifting wrt A1 and A2}.
\end{proof} 

\begin{proof}[Proof of Theorem~\ref{thm:B}]
This follows immediately from the 
fact established in the proof 
above that the class of acyclic cofibrations 
in the Joyal model structure on $\sSet$ is 
equal to the class $\calC_{\calW}$.  
\end{proof}

\begin{remark} 
\label{rem:ingredients}
In addition to the results from 
this paper, the construction of 
the Joyal model structure that 
we have given here depends on the 
results proven in the first 13 pages of \cite{J1} 
(the key result being Theorem 1.3 of 
that paper), Corollary 2.3.6.5 of 
\cite{HTT}, and the fact that a 
left fibration with contractible 
fibers is a trivial Kan fibration
(Lemma 2.1.3.4 of \cite{HTT}).  
\end{remark} 

\section{Technical lemmas on inner anodyne maps} 
\label{sec:technical lemmas}

In this section we collect together some 
technical lemmas on inner anodyne maps 
that we use at various points in the 
paper.  A key ingredient in the proof 
of these lemmas is the right cancellation 
property for inner anodyne maps (Lemma~\ref{lem:right cancellation inner}).  

The following lemma is straightforward.  
We shall need it for the proof 
of Proposition~\ref{prop:key tech prop2}. 

\begin{lemma}
\label{lem:joins and inner anodynes} 
If $u\colon A\to B$ is an inner anodyne
map between simplicial sets, 
then the induced map $u\star \id_C\colon 
A\star C\to B\star C$ is inner anodyne 
for any simplicial set $C$.  
\end{lemma} 

\begin{proof} 
In the pushout diagram 
\[
\begin{tikzcd} 
A \arrow[r] \arrow[d,"u"] & A\star C \arrow[d] \\ 
B \arrow[r] & A\star C\cup B 
\end{tikzcd} 
\]
both vertical maps are inner anodyne.  The 
canonical map $A\star C\cup B\to B\star C$ 
is inner anodyne by Lemma 2.1.2.3 of \cite{HTT}.  
Therefore the composite map 
\[
A\star C\to A\star C\cup B\to B\star C 
\]
is inner anodyne.  
\end{proof} 

The next lemma from \cite{S2} 
is a straightforward application 
of the right cancellation property.  

\begin{lemma}[\cite{S2}] 
\label{lem:pushout lemma for ia}
Suppose given a commutative diagram 
of simplicial sets of the 
form 
\[
\begin{tikzcd} 
A_1 \arrow[d] & 
A_0 \arrow[l] \arrow[d] 
\arrow[r] & A_2 \arrow[d] \\ 
B_1 & \arrow[l] B_0 
\arrow[r] & B_2 
\end{tikzcd} 
\]
in which at least one of the 
squares is a pullback.  
If all of the vertical 
arrows are inner anodyne 
then the induced map 
\[
A_1\cup_{A_0}A_2 \to 
B_1\cup_{B_0}B_2 
\]
is also inner anodyne.  
\end{lemma} 
 
We have following refinement 
of Lemma 5.4.5.10 from \cite{HTT}.  

\begin{lemma} 
\label{lem:5.4.5.10}
Suppose that $A$, $B$ and $C$ are 
simplicial sets and that $B$ is weakly 
contractible.  Then the canonical map 
\[
A\star B\cup B\star C\subseteq A\star 
B\star C
\]
is inner anodyne.  
\end{lemma} 

\begin{proof} 
One can check that the 
steps (1) -- (7) of Lurie's proof 
of Lemma 5.4.5.10 go through with 
monic categorical equivalences 
replaced by inner anodyne maps 
using Lemma~\ref{lem:pushout lemma for ia} 
and the right cancellation 
property for inner anodyne 
maps (Lemma~\ref{lem:right cancellation inner}).  
\end{proof}

We shall make use of the following generalization 
of Lemma~\ref{lem:5.4.5.10} in 
the proof of Proposition~\ref{sec:proof of gen lurie lemma}.

\begin{lemma} 
\label{lem:variant of 5.4.5.10}
Let $A_0,\ldots,A_n$ be simplicial sets where 
$n\geq 2$.  Suppose that $0<i<n$ and that 
each of the simplicial sets 
$A_1,\ldots,A_{n-1}$ is weakly contractible.  
Then 
the inclusion 
\begin{equation}
\label{eq:variant of 5.4.5.10}
\bigcup_{0\leq j\leq n,j\neq i} 
A_0\star \cdots \star \widehat{A}_j\star \cdots 
\star A_n\to A_0\star \cdots \star A_n 
\end{equation} 
is inner anodyne, where the hat indicates omission.  
\end{lemma}

\begin{proof} 
The proof is by induction on $n$.  
Lemma~\ref{lem:5.4.5.10} gives 
the statement when $n=2$.   
Let $0<i<n$ where $n\geq 3$.  Without loss of generality 
we may suppose by a duality argument that $i<n-1$.  
We have a pushout diagram of the form 
\[
\begin{tikzcd} 
\displaystyle{\bigcup_{0\leq j\leq n-1,j\neq i}}
A_0\star \cdots \star \widehat{A}_j\star \cdots 
\star A_{n-1} \arrow[r] \arrow[d] & 
\displaystyle{\bigcup_{0\leq j\leq n-1,j\neq i}}
A_0\star \cdots \star \widehat{A}_j\star \cdots 
\star A_{n}\arrow[d]                                 \\ 
A_0\star \cdots \star A_{n-1} \arrow[r] &
\displaystyle{\bigcup_{0\leq j\leq n,j\neq i}}
A_0\star \cdots \star \widehat{A}_j\star \cdots 
\star A_{n}
\end{tikzcd} 
\]
The left hand vertical map is inner anodyne by 
the inductive hypothesis and hence so is the right hand 
vertical map.  The map 
\[
\bigcup_{0\leq j\leq n-1,j\neq i} 
A_0\star \cdots \star \widehat{A}_j\star \cdots \star 
A_n \to (A_0\star \cdots \star A_{n-1})\star A_n 
\]
is inner anodyne by Lemma~\ref{lem:joins and inner anodynes}.  
The result then follows 
by the right cancellation property of inner anodyne maps 
(Lemma~\ref{lem:right cancellation inner}).
\end{proof} 

\begin{remark}
Observe that if one of the simplicial sets 
$A_1,\ldots A_{i-1}, A_{i+1},\ldots, 
A_{n-1}$ in the statement of 
Lemma~\ref{lem:variant of 5.4.5.10} 
is empty then the inclusion~\eqref{eq:variant of 5.4.5.10} 
is an equality.  Therefore the hypotheses 
of Lemma~\ref{lem:variant of 5.4.5.10} 
can be relaxed to cover the cases where 
each of the simplicial sets $A_1,\ldots, 
A_{i-1},A_{i+1},\ldots, A_{n-1}$ are 
either empty or weakly contractible and 
$A_i$ is weakly contractible.  
\end{remark} 

\begin{corollary} 
\label{cor:variant of 5.4.5.10}
Let $A_0,\ldots, A_n$ be simplicial sets where $n\geq 5$.  If each 
of $A_1,\ldots, A_{n-1}$ is weakly contractible then 
for any $i+1<j<n$ the inclusion 
\begin{equation}
\label{eq:corr to variant of 5.4.5.10}
\bigcup^n_{\stackrel{\scriptstyle{k=0}}{k\neq i,j}} A_0\star \cdots 
\star \widehat{A}_k\star \cdots \star A_n 
\subseteq A_0\star \cdots \star A_n 
\end{equation} 
is inner anodyne.  
\end{corollary} 

\begin{proof} 
This follows easily from 
Lemma~\ref{lem:variant of 5.4.5.10} 
and Lemma 2.1.2.3 of \cite{HTT}.
\end{proof} 

\begin{remark} 
\label{rem:variant of corr to 5.4.5.10}
If one of the simplicial sets $A_1,\ldots, A_{i-1},A_{i+1},\ldots, 
A_{n-1}$ in the statement of Corollary~\ref{cor:variant of 5.4.5.10} 
is empty then the inclusion~\eqref{eq:corr to variant of 5.4.5.10} 
is an isomorphism.  Therefore the hypotheses of 
Corollary~\ref{cor:variant of 5.4.5.10} can be relaxed to 
cover the cases where each of the simplicial 
sets $A_1,\ldots, A_{i-1},A_{i+1},\ldots, A_{n-1}$ 
are either empty or weakly contractible and 
$A_i$ is weakly contractible.   
\end{remark}

Finally, we have the following lemma which  
we have made use of above in the proof of 
Proposition~\ref{prop:key prop}.  

\begin{lemma}
\label{lem:ladders of inner anodynes}
Suppose given a commutative diagram in $\sSet$ of the 
form 
\[
\begin{tikzcd} 
A_0 \ar[r,hook] \ar[d] & A_1 \ar[r,hook] \ar[d] & A_2 \ar[r,hook] \ar[d] & \cdots \\ 
B_0 \ar[r,hook] & B_1 \ar[r,hook] & B_2 \ar[r,hook] & \cdots 
\end{tikzcd} 
\]
in which the horizontal arrows are inclusions.  If the map 
$A_n\to B_n$ is inner anodyne for all $n\geq 0$, then the 
induced map 
\[
\varinjlim_n A_n \to \varinjlim_n B_n 
\]
is inner anodyne.  
\end{lemma} 

\begin{proof} 
A standard argument shows that 
it suffices to prove that the canonical map 
\[
A_{n+1}\cup B_n\to B_{n+1} 
\]
induced from the commutative diagram 
\[
\begin{tikzcd} 
A_n \ar[r] \ar[d] & A_{n+1} \ar[d] \\ 
B_{n} \ar[r] & B_{n+1} 
\end{tikzcd} 
\]
is inner anodyne.  But this is a straightforward 
consequence of the right cancellation 
property (Lemma~\ref{lem:right cancellation inner}) 
of the class of inner anodyne maps, using 
the fact that the vertical maps in the diagram 
above are inner anodyne.  
\end{proof}

\section{Proof of Proposition~\ref{prop:key tech prop2}}
\label{sec:proof of gen lurie lemma}

In this section we prove Proposition~\ref{prop:key tech prop2}.  
For the proof we will need the 
following straightforward lemma 
(in fact, for the application we have 
in mind --- see the proof of 
Proposition~\ref{prop:descent pre-fibrant simp sets} --- 
we could use an ad-hoc argument 
which does not refer to cocartesian edges).  

\begin{lemma} 
\label{lem:cocart lemma}
Let $p\colon X\to S$ be an inner fibration between simplicial sets.  
Suppose given a commutative diagram of the form 
\[
\begin{tikzcd} 
\Delta^{\set{0,2,\ldots, n}} \cup 
\Delta^{\set{0,1}} \arrow[d] \arrow[r,"u"] & X \arrow[d,"p"] \\ 
\Delta^{\set{0,\ldots, n}} \arrow[r] \arrow[ur,dashed] & S 
\end{tikzcd} 
\]
If the edge $u|\Delta^{\set{0,1}}$ is $p$-cocartesian then 
the indicated diagonal filler exists.  
\end{lemma}

\begin{proof} 
If $u|\Delta^{\set{0,1}}$ is $p$-cocartesian, then we can find 
the indicated diagonal filler in the diagram 
\[
\begin{tikzcd} 
\Delta^{\set{0,2,\ldots,n}}\cup \Delta^{\set{0,1}} 
\arrow[d] \arrow[r,"u"] & X \arrow[dd,"p"] \\ 
\Delta^{\set{0,2,\ldots,n}}\cup 
\Delta^{\set{0,1,2}} \arrow[ur,dashed] \arrow[d] & \\ 
\Delta^{\set{0,\ldots,n}} \arrow[r] & S 
\end{tikzcd} 
\]
A straightforward argument shows that the map 
$\Delta^{\set{0,2,\ldots,n}}\cup 
\Delta^{\set{0,1,2}}\to \Delta^{\set{0,\ldots,n}}$ 
is inner anodyne.  This suffices to complete the proof.  
\end{proof}

We now prove Proposition~\ref{prop:key tech prop2}.  As 
noted earlier, this proposition is a generalization 
of Lemma 3.1.2.4 from \cite{HA}.    
We recall the statement.  

\begin{proposition*} 
Let $p\colon \calC\to \Lambda^n_i$ be an inner fibration where $n\geq 3$.  
Let $q$ denote the projection 
\[
q\colon \Delta^{\set{1,\ldots,n}}\times_{\Lambda^n_i} \calC\to 
\Delta^{\set{1,\ldots,n}}. 
\] 
Suppose that the following conditions are satisfied:  
\begin{enumerate}[label=(\roman*)]
\item for every vertex $x\in \calC_j$, $1\leq j\leq n-2$, there 
exists a $q$-cocartesian morphism $f\colon x\to y$ where 
$y\in \calC_{j+1}$; 

\item for every vertex 
$y\in \calC_j$, $2\leq j\leq n-1$ there exists a 
$q$-cartesian morphism $f\colon x\to y$ where 
$x\in \calC_{j-1}$. 
\end{enumerate}
Then $p\colon \calC\to \Lambda^n_i$ satisfies 
inner anodyne descent with respect to the inclusion 
$\Lambda^n_i\subseteq \Delta^n$.  
\end{proposition*} 

\begin{proof} 
The proof that we give is an adaptation of 
Lurie's proof of Lemma 3.1.2.4 of \cite{HA}.  
The objective is to construct a sequence 
\[
\calC = \calC(0)\subseteq \calC(1)\subseteq \calC(2)\subseteq \cdots 
\]
in $(\sSet)_{/\Delta^n}$ 
where each inclusion $\calC(m)\subseteq \calC(m+1)$ is 
inner anodyne and where the following conditions are 
satisfied: 

\begin{enumerate}[label=(\roman*)]
\item if $\Lambda^r_k\to \calC(m)$ is a map in 
$(\sSet)_{/\Delta^n}$ with $0<k<r$, then the 
composite map $\Lambda^r_k\to \calC(m)\to \calC(m+1)$ factors 
through the inclusion $\Lambda^r_k\subseteq \Delta^r$ in 
$(\sSet)_{/\Delta^n}$; 

\item for each $m\geq 0$ the diagram 
\[
\begin{tikzcd} 
\calC \arrow[d,"p"] \arrow[r] & \calC(m) \arrow[d] \\ 
\Lambda^n_i \arrow[r] & \Delta^n 
\end{tikzcd} 
\]
is cartesian.  
\end{enumerate}

The construction proceeds by induction on $m\geq 0$.  Assuming that 
$\calC(m)$ has been defined, let $A$ denote the set of maps 
$\alpha\colon \Lambda^r_k\to \calC(m)$ in 
$(\sSet)_{/\Delta^n}$ with $0<k<r$.  
It suffices to construct, for 
each map $\alpha$ in $A$, a 
simplicial set $\calC(m,\alpha)$ in $(\sSet)_{/\Delta^n}$  
with the following 
properties: 
\begin{enumerate}[label=(\roman*)]
\setcounter{enumi}{2}
\item the inclusion $\calC(m)\subseteq \calC(m,\alpha)$ is inner 
anodyne; 

\item the composite map $\Lambda^r_k\xrightarrow{\alpha} 
\calC(m)\subseteq \calC(m,\alpha)$ can be extended to 
an $r$-simplex of $\calC(m,\alpha)$; 

\item the diagram 
\[
\begin{tikzcd} 
\calC \arrow[d,"p"] \arrow[r] & \calC(m,\alpha) \arrow[d] \\ 
\Lambda^n_i \arrow[r] & \Delta^n 
\end{tikzcd} 
\]
is cartesian.     
\end{enumerate} 

Assuming that the simplicial sets $\calC(m,\alpha)$ 
have been constructed satisfying the properties 
(iii) -- (v) above, we set $\calC(m+1)$ to be 
the coproduct $\sqcup_{\alpha\in A}\calC(m,\alpha)$ 
in the category $(\sSet)_{\calC(m)//\Delta^n}$.  It 
is not hard to show that the canonical map 
$\calC(m)\to \calC(m+1)$ is inner anodyne and that 
the properties (i) and (ii) above are satisfied.  

Suppose given a map $\alpha\colon \Lambda^r_k\to \calC(m)$ 
in $A$; let $\alpha_0\colon \Lambda^r_k\to \Delta^n$ denote 
the composite map.  Without loss of generality we may 
suppose that the image of $\alpha_0$ contains 
$\partial_i\Delta^n = \Delta^{\set{0,\ldots,i-1,i+1,\ldots,n}}
\subseteq \Delta^n$.  Note that 
$\alpha_0(0) = 0$ and $\alpha_0(r) =n$.  Let 
$\bar{\alpha}_0\colon \Delta^r\to \Delta^n$ 
denote the unique extension of $\alpha_0$.  
Without loss of generality we may suppose that 
the canonical map 
\[
\Lambda^r_k\times_{\Delta^n}\Lambda^n_i \hookrightarrow 
\Delta^r\times_{\Delta^n}\Lambda^n_i 
\]
is not surjective (if the canonical map is 
surjective then we may take $\calC(m,\alpha) = 
\calC(m)\cup_{\Lambda^r_k}\Delta^r$ and check 
that the properties (iii) -- (v) above 
are satisfied). 

Therefore, the simplex 
$\partial_k\Delta^r = \Delta^{\set{0,\ldots,k-1,k+1,\ldots,r}}$ is mapped 
into $\Lambda^n_i$ by $\bar{\alpha}_0$.  
Since $\bar{\alpha}_0(0) =0$ and $\bar{\alpha}_0(r) = n$, 
the simplex $\partial_k\Delta^{r}$ must be 
mapped into a simplex $\partial_j\Delta^{n} = 
\Delta^{\set{0,\ldots,j-1,j+1,\ldots,n}}$ of 
$\Lambda^n_i$ by $\bar{\alpha}_0$. Note that 
$0<j<n$ and $j\neq i$.  Under the given hypotheses 
we may suppose, by a duality argument, that without 
loss of generality $j\leq n-2$.    
Note also that we must have $\bar{\alpha}_0^{-1}\set{j} = \set{k}$, 
since otherwise the vertex $j$ belongs to the image of 
$\partial_k\Delta^{r}$ under $\bar{\alpha}_0$ (recall that the 
image of $\bar{\alpha}_0$ contains $\partial_i\Delta^n$).

Write $A_l = \bar{\alpha}_0^{-1}\set{l}$ for 
each $l=0,\ldots, n$.  Note that 
$A_j = \set{x}$, where $x$ denotes the $k$-th vertex of 
$\Delta^r$.  It follows that we have a canonical isomorphism  
\[
\Delta^r = A_0\star  \cdots 
\star A_n.  
\]
Note that $A_0$ and $A_n$ are non-empty and that each 
$A_l$ (if non-empty) is isomorphic to a simplex.  It is 
convenient to introduce the following shorthand notation: 
if $\set{i_1,\ldots, i_p}$ is a subset of $\set{0,\ldots, n}$ then 
we write 
\[
A_{i_1,\ldots, i_p} = A_{i_1}\star \cdots \star A_{i_p}.  
\]
Let $X$ denote the image of the vertex $x$ under the map 
$\alpha$, and choose a $q$-cocartesian arrow $f\colon X\to Y$
in $\calC$, where $Y\in \calC_{j+1}$.  
Since $f$ is $q$-cocartesian, Lemma~\ref{lem:cocart lemma} 
shows that we can find the indicated diagonal filler $\beta$ 
making the diagram 
\[
\begin{tikzcd} 
\set{x}\star \set{y}\cup \set{x}\star A_{j+1,\ldots, n} 
 \arrow[d] \arrow[r] & \calC    \arrow[d,"p"]                  \\ 
\set{x}\star \set{y}\star A_{j+1,\ldots, n} 
 \arrow[ur,dashed,"\beta"'] \arrow[r] & \Lambda^n_i 
\end{tikzcd} 
\]
commute, where the upper horizontal map is induced by $f$ and $\alpha$.  

Let $K_3 = A_0\star \cdots \star 
A_j\star \set{y}\star A_{j+1}\star \cdots \star A_n$ so that 
$K_3$ is isomorphic to $\Delta^{r+1}$.  By an 
abuse of notation let us denote by $\overline{\alpha}_0\colon K_3\to 
\Delta^n$ the canonical extension of 
$\overline{\alpha}_0\colon \Delta^r\to \Delta^n$ 
with $\bar{\alpha}_0(y) = j+1$.  

Our aim is to construct (following Lurie's proof of Lemma 3.1.2.4 in \cite{HA}) a sequence of simplicial sets 
\[
K_0\subseteq K_1 \subseteq K_2 \subseteq K_3 
\]
where each inclusion is a categorical equivalence, 
together with maps 
$K_i\to \calC(m)$ for $i=0,1,2$, compatible with the inclusions and with the maps 
$\alpha$ and $\beta$.  
We will then set 
\[
\calC(m,\alpha) = \calC(m)\cup_{K_2} K_3
\]
and show that $\calC(m,\alpha)$ has the desired properties.

To begin with, we define the 
subcomplex $K_0$ of $K_3$ to be equal to 
the inverse image of $\Lambda^n_i\cap \Lambda^n_j\cap \Lambda^n_{j+1}$ 
in $K_3$ under the map $\overline{\alpha}_0$.  In other 
words 
\[
K_0 = \bigcup^{j-1}_{l=0, l \neq i}A_{0,\ldots, \hat{l},\ldots, j} 
\star \set{y}\star A_{j+1,\ldots, n} \cup 
\bigcup^n_{l=j+2, l\neq i}A_{0,\ldots, j}\star \set{y}\star 
A_{j+1,\ldots, \hat{l},\ldots, n}
\]
where the hats indicate omission.  
Next, we show that there is a map $\phi\colon K_0\to \calC$ which 
is compatible with $\alpha$ and $\beta$.  Define a sequence of 
inclusions 
\[
K_0^{(0)} \subseteq K_0^{(1)} \subseteq \cdots 
\subseteq K_0^{(j-1)} = K_0^{(j)} = K_0^{(j+1)} \subseteq \cdots 
\subseteq K_0^{(n)} = K_0 
\]
by setting 
\[
K_0^{(t)} = \bigcup^t_{l=0,l\neq i} A_{0,\ldots, \hat{l},\ldots,j} 
\star \set{y}\star A_{j+1,\ldots, n} 
\]
for $0\leq t\leq j-1$ and 
\[
K_0^{(t)} = \bigcup^t_{l=j+2, l\neq i} A_{0,\ldots, j}\star 
\set{y}\star A_{j+1,\ldots, \hat{l},\ldots ,n} 
\]
for $j+2\leq t\leq n$.  

For each $t=0,\ldots n$ we will construct a 
compatible sequence of maps $\phi_t\colon K_0^{(t)}\to 
\calC(m)$, each of which is compatible with $\alpha$ and 
$\beta$.  The construction is by induction on $t$.  

For the case $t=0$, we observe that the inclusion 
\[
A_j\star \set{y}\star A_{j+1,\ldots,n} 
\cup A_{1,\ldots ,j}\star A_{j+1, 
\ldots ,n}                                 
\subseteq A_{1,\ldots, j}\star \set{y}\star A_{j+1,
\ldots ,n} = K_0^{(0)} 
\] 
is inner anodyne, since the inclusion $A_j\subseteq 
A_{1, \ldots ,j}$ is right anodyne (Lemma 2.1.2.3 
of \cite{HTT}).    
Therefore, since $p\colon \calC\to \Lambda^n_i$ 
is an inner fibration, we can choose the indicated diagonal 
filler $\phi_0$ making the diagram 
\[
\begin{tikzcd} 
A_j\star \set{y}\star A_{j+1,\ldots,n} 
\cup A_{1,\ldots ,j}\star A_{j+1, 
\ldots ,n}   \arrow[d] \arrow[r] & \calC  \arrow[d]     \\ 
K_0^{(0)}  
\arrow[ur,dashed,"\phi_0"'] \arrow[r] & \Lambda^n_i
\end{tikzcd} 
\]
commute, where the upper horizontal map is induced by $\alpha$ and 
$\beta$.  

Suppose that a map $\phi_t\colon K_0^{(t)}\to \calC$ 
has been constructed, whose restriction to $K_0^{(r)}$ 
is equal to $\phi_r$ for all $r\leq t$.    
We construct a map $\phi_{t+1}\colon K_0^{(t+1)}\to \calC$ 
whose restriction to $K_0^{(t)}$ is equal to $\phi_t$.    

If $t<j$ then the inclusion $K_0^{(t)}\cup A_{0,\ldots, \widehat{t+1},\ldots, j} 
\star A_{j+1,\ldots, n} \subseteq K_0^{(t+1)}$ 
is obtained as the pushout 
\[
\begin{tikzcd} 
\left(K_0^{(t)}\cup A_{0,\ldots, \widehat{t+1},\ldots, j} 
\star A_{j+1,\ldots, n}\right) \cap A_{0,\ldots,\widehat{t+1},\ldots ,j} 
\star\set{y}\star A_{j+1,\ldots, n} \arrow[r] \arrow[d] 
& K_0^{(t)} \arrow[d]                                   \\ 
A_{0,\ldots,\widehat{t+1},\ldots ,j} 
\star\set{y}\star A_{j+1,\ldots, n} \arrow[r] & 
K_0^{(t+1)} 
\end{tikzcd} 
\]
An application of Lemma~\ref{lem:variant of 5.4.5.10} combined  
with Lemma~\ref{lem:joins and inner anodynes} shows that 
the left hand vertical map is inner anodyne, and hence 
that the inclusion $K_0^{(t)}\cup A_{0,\ldots, \widehat{t+1},\ldots, j} 
\star A_{j+1,\ldots, n}\subseteq K_0^{(t+1)}$ is inner 
anodyne.  Therefore, the indicated diagonal 
filler $\phi_{t+1}$ exists in the diagram 
\[
\begin{tikzcd} 
K_0^{(t)}\cup A_{0,\ldots, \widehat{t+1},\ldots, j} 
\star A_{j+1,\ldots, n} \arrow[r] \arrow[d] & \calC \arrow[d] \\ 
K_0^{(t+1)} \arrow[r] \arrow[ur,dashed,"\phi_{t+1}"'] & \Lambda^n_i 
\end{tikzcd} 
\]
where the upper horizontal map is induced by $\phi_t$ and $\alpha$.  
If $t>j+1$ then the argument is analogous, except that either 
Lemma~\ref{lem:variant of 5.4.5.10} or Corollary~\ref{cor:variant of 5.4.5.10} 
is combined with Lemma~\ref{lem:joins and inner anodynes} to show that 
the inclusion $K_0^{(t)}\cup A_{0,\ldots,j}\star A_{j+1,
\ldots, \widehat{t+1},\ldots, n}\subseteq K_0^{(t+1)}$ 
is inner anodyne.  
This completes the inductive step, and hence the construction of the 
map $\phi\colon K_0\to \calC$.

We now construct a subcomplex $K_1\subseteq K_3$
by 
\[
K_1 = K_0\cup \Lambda^r_k.  
\]
Hence $K_1$ forms part of a pushout diagram  
\[
\begin{tikzcd} 
K_0\cap \Lambda^r_k  
\arrow[d] \arrow[r] & K_0 \arrow[d]                \\ 
\Lambda^r_k \arrow[r] & K_1.  
\end{tikzcd} 
\]
The inclusion $K_0\cap \Lambda^r_k \subseteq \Delta^r$ 
is equal to the map 
\[
\bigcup^{n}_{\stackrel{\scriptstyle{l=0}}{l\neq i,j,j+1}}A_{0,\ldots, \hat{l},\ldots, n} 
\subseteq \Delta^r.  
\]
Since $A_j\star A_{j+1}$ is weakly contractible it follows 
from Lemma~\ref{lem:variant of 5.4.5.10} that $K_0\cap 
\Lambda^r_k\subseteq \Delta^r$ is inner anodyne.  Therefore, 
$K_0\cap \Lambda^r_k\subseteq \Lambda^r_k$ is a 
categorical equivalence, by the 2-out-of-3 property 
of categorical equivalences 
(Remark~\ref{rem:2 out of 3 cat equiv}).  It follows 
(Lemma~\ref{lem:cat equivs stable under po}) that the inclusion 
$K_0\subseteq K_1$ is a categorical equivalence.   
Let $\psi\colon K_1\to \calC(m)$ denote the unique map 
extending $\alpha\colon \Lambda^r_k\to \calC(m)$ and 
$\phi\colon K_0\to \calC$.     

Finally, we construct the subcomplex $K_2\subseteq K_3$ 
by 
setting 
\[
K_2 = K_1 \cup A_{0,\ldots, j-1} \star \set{y}\star A_{j+1, 
\ldots ,n}.  
\]
The inclusion $K_1\subseteq K_2$ is a pushout of the inclusion 
\[
K_0\cap A_{0, \ldots,j-1}\star \set{y}\star A_{j+1, 
\ldots,n} 
\subseteq A_{0,\ldots,j-1}\star \set{y}\star A_{j+1, 
\ldots,n} 
\]
since $A_{0,\ldots,j-1}\star \set{y}\star A_{j+1, 
\ldots,n} \cap \Lambda^r_k = \emptyset$ (recall that $\Lambda^r_k$ is 
the union of the faces of $\Delta^r$ which contain the vertex $x$).  An argument 
analogous to the argument used above to show that 
$K_0\cap \Lambda^r_k\subseteq \Delta^r$ is inner anodyne 
shows that the inclusion $K_0\cap A_{0,\ldots,
j-1}\star \set{y}\star A_{j+1,\ldots,n} 
\subseteq A_{0, \ldots,j-1}\star \set{y}\star 
A_{j+1,\ldots ,n}$ is inner anodyne.  Since 
$p\colon \calC\to \Lambda^n_i$ is an inner 
fibration, the restriction of $\phi$ extends along the 
inclusion $K_0\cap A_{0,\ldots,
j-1}\star \set{y}\star A_{j+1,\ldots,n} 
\subseteq A_{0, \ldots,j-1}\star \set{y}\star 
A_{j+1,\ldots ,n}$ to define a map 
\[
\chi\colon A_{0, \ldots,j-1}\star \set{y}\star 
A_{j+1,\ldots ,n}\to \calC.  
\]
There is a unique map $K_2\to 
\calC(m)$ extending the map $\psi$ and the map $\chi$. 
We define 
\[
\calC(m,\alpha) = \calC(m)\cup_{K_2} K_3 
\]
An easy 2-out-of-3 argument shows that the inclusion 
$K_2\subseteq K_3$ is a categorical equivalence; since 
$K_3$ is an $\infty$-category it follows by 
Lemma~\ref{lem:easy version of Joyals conj} that 
$K_2\subseteq K_3$ is inner anodyne.  Hence the inclusion 
$\calC(m)\subseteq \calC(m,\alpha)$ is inner anodyne.  

To complete the construction we need to show that the induced map 
\[
\calC\subseteq \calC(m,\alpha)\times_{\Delta^n}\Lambda^n_i 
\]
is an isomorphism.  It suffices to prove that the induced 
map 
\[
K_2\times_{\Delta^n}\Lambda^n_i \subseteq 
K_3\times_{\Delta^n}\Lambda^n_i 
\]
is an isomorphism, which is clear.  
\end{proof}

\section{Proof of Lemma~\ref{lem:bergners lemma}}
\label{sec:proof of bergners lemma}

Our objective in this section is to prove 
Lemma~\ref{lem:bergners lemma}.  Recall 
that this is the analog for simplicial sets 
of Lemma 2.4 in \cite{Bergner}.  
We will need some preliminary lemmas.    

Firstly, we shall need the following result 
asserting that inner anodyne maps 
pullback along Kan fibrations to inner anodyne 
maps (this could be deduced from 
Proposition 3.3.1.3 of \cite{HTT} 
however the proof of the latter result 
is much more difficult).  

\begin{lemma} 
\label{lem:pullback of inner anodyne along Kan fibn}
Suppose that $p\colon X\to S$ is a Kan fibration 
between simplicial sets.  Then the induced map 
\[
A\times_S X\to B\times_S X 
\]
is inner anodyne for any inner anodyne map 
$A\to B$ in $(\sSet)_{/S}$.  
\end{lemma} 

\begin{proof} 
Let $\calA$ be the class of all monomorphisms $u\colon A\to B$ in 
$(\sSet)_{/S}$ with the property that the induced map 
\[
A\times_S X\to B\times_S X 
\]
is inner anodyne.  We will prove that $\calA$ contains the 
class of all inner anodyne maps in $(\sSet)_{/S}$.  
The functor $(\sSet)_{/S}\to (\sSet)_{/X}$ defined 
by pullback along the map $p\colon X\to S$ is a left 
adjoint and hence sends saturated classes to saturated 
classes.  Therefore it suffices to prove the statement 
in the special case that $u\colon A\to B$ is an 
inner horn inclusion $\Lambda^n_i\subseteq \Delta^n$ 
in $(\sSet)_{/S}$.  

Suppose given then a Kan fibration $p\colon X\to \Delta^n$.  
We need to prove that the induced map $\Lambda^n_i\times_{\Delta^n}X 
\to X$ is inner anodyne.  Since the class of inner anodyne 
maps is closed under retracts we may assume without 
loss of generality that the Kan fibration $X\to \Delta^n$ is minimal.  
But then a classical theorem asserts that $X\to \Delta^n$ 
is a trivial bundle and hence there is an isomorphism 
$X\simeq \Delta^n\times K$ for some Kan complex $K$. 
But then the induced map $\Lambda^n_i\times_{\Delta^n}X 
\to X$ is isomorphic to the map $\Lambda^n_i\times K 
\to \Delta^n\times K$ which is clearly inner 
anodyne.  
\end{proof}

Most of the work involved in proving 
Lemma~\ref{lem:bergners lemma} is contained 
in proving the following result.

\begin{lemma}
\label{lem:GJ lemma}
Suppose that $B$ is a simplicial set 
with two vertices $0$ and $1$ and suppose 
that $\set{0}\to B$ is a categorical equivalence.  
There is a countable subcomplex 
$D\subseteq B$ satisfying the conditions that 
(i) $\set{0}\to D$ is a 
categorical equivalence and (ii) the vertices 
$0$ and $1$ both belong to $D$.  
\end{lemma}


The proof of Lemma~\ref{lem:GJ lemma} 
will require some preparation.  Let $\calB$ denote 
the subcategory of $(\sSet)_{\sk_2B//B}$ 
whose objects are the subcomplexes $B'$ of 
$B$ such that $\sk_2(B') = \sk_2(B)$, and 
whose morphisms are the inclusions between them.  

\begin{construction}
We construct a functor 
\[
X\colon \calB \to (\sSet)_{\set{0}//B}^{[1]} 
\]
from $\calB$ to the arrow category of 
$(\sSet)_{\set{0}//B}$ 
which sends a subcomplex $B'\in \calB$ to 
a diagram 
\[
\set{0}\to X(B')\to B' 
\]
in which $\set{0}\to X(B')$ is a categorical 
equivalence and $X(B')\to B'$ is a Kan fibration 
which is natural in $B'\in \calB$.  
Moreover the functor $X$ is compatible with 
filtered colimits in $\calB$.    
\end{construction}

The construction proceeds as follows.  
It is directly inspired by the proof of 
Lemma X.2.8 from \cite{GJ}.    
The small object 
argument for the set of inner 
horn inclusions $\Lambda^n_i 
\subseteq \Delta^n$ with $0<i<n$ 
gives rise to a functor 
\[
R\colon \sSet\to \sSet
\]
together with a natural transformation 
$\mathrm{id}_{\sSet}\to R$ such that 
$R(S)$ is an $\infty$-category for every 
simplicial set $S$ and the component 
$S\to R(S)$ of the natural transformation 
is an inner anodyne map for 
every simplicial set $S$.  Moreover 
the functor $R$ is compatible with 
filtered colimits.  

The functor $R$ induces a functor 
$R\colon \calB\to (\sSet)_{/R(B)}$ such that 
for every subcomplex $B'\in \calB$ there 
is a commutative diagram, natural in $B'$, 
of the form  
\[
\begin{tikzcd}
B' \arrow[d] \arrow[r] & R(B') \arrow[d] \\ 
B \arrow[r] & R(B) 
\end{tikzcd} 
\]
in which the horizontal maps are inner anodyne.  
Note that since $\sk_2B' = \sk_2B$ for every 
subcomplex $B'$ in $\calB$, and $\h(B)$ is a 
(contractible) groupoid, it follows that 
$R(B')$ is a Kan complex for every $B'\in \calB$ 
(Corollary 1.4 of \cite{J1}).  
In particular $R(B)$ is a Kan complex.  

Using the small object 
argument for the set of horn inclusions 
$\Lambda^n_i\subseteq \Delta^n$ with $0\leq i\leq n$, we see that for 
every $B'\in \calB$, there exists a commutative diagram, 
natural in $B'$, of the form 
\begin{equation} 
\label{eq:diagram for X(B')}
\begin{tikzcd} 
\set{0} \arrow[d] \arrow[r] & T_{B'}(\set{0})\arrow[d] \\ 
B' \arrow[r] & R(B') 
\end{tikzcd} 
\end{equation}
in which $\set{0}\to T_{B'}(\set{0})$ is a homotopy equivalence 
and $T_{B'}(\set{0})\to R(B')$ is a Kan fibration.  
Since both $\set{0}$ and $T_{B'}(\set{0})$ are Kan complexes, 
it follows that the map $\set{0}\to T_{B'}(\set{0})$ 
is a categorical equivalence 
(Lemma~\ref{lem:hty equiv between Kan complexes}).  

Since the natural transformation $B'\to R(B')$ is compatible with filtered colimits 
in $B'$ it follows that the diagram~\eqref{eq:diagram for X(B')} 
is compatible with filtered colimits in $B'$.

Since each map $B'\to R(B')$ is inner anodyne, it 
follows (Lemma~\ref{lem:pullback of inner anodyne along Kan fibn}) 
that the induced maps $B'\times_{R(B')}T_{B'}(\set{0})\to 
T_{B'}(\set{0})$ is inner anodyne, and hence that the 
map $\set{0}\to B'\times_{R(B')}T_{B'}(\set{0})$ is a 
categorical equivalence by the 2-out-of-3 property 
for categorical equivalences.  

Define $X(B') := B'\times_{R(B')}T_{B'}(\set{0})$.  
Then the following statements are true for every 
$B'\in \calB$: 

\begin{itemize}
\item The induced map $X(B')\to B'$ is a Kan fibration which is natural in $B\in \calB$; 
\item The map $\set{0}\to X(B')$ is a categorical 
equivalence; 
\item The map $\set{0}\to B'$ is a categorical equivalence 
if and only if the map $X(B')\to B'$ is a trivial 
Kan fibration.  
\end{itemize} 

Moreover, since $R(B')$ and $T_{B'}(\set{0})$ are natural 
in $B'$ and compatible with filtered colimits in $\calB$, 
the construction $X(B')$ extends to define a functor 
\[
X\colon \calB\to (\sSet)_{\set{0}//B}^{[1]} 
\]
which is compatible with filtered colimits in $\calB$.


Observe that if $B'\in \calB$ 
has only countably many non-degenerate 
simplices then $R(B')$ has only countably 
many non-degenerate simplices.  Therefore 
the same is true for $T_{B'}(\set{0})$ and 
hence for $X(B')$.  

We can now prove Lemma~\ref{lem:GJ lemma}.

\begin{proof}[Proof of Lemma~\ref{lem:GJ lemma}]
Let $D_0 = \sk_2B$.  Consider all 
lifting problems for diagrams 
of the form 
\begin{equation}
\label{eq:analog of GJ diagram}
\begin{tikzcd} 
\partial\Delta^n \arrow[d] \arrow[r] & X(D_0) \arrow[r] & X(B) \arrow[d] \\ 
\Delta^n \arrow[urr,dashed] \arrow[r] & D_0 \arrow[r,hook] & B 
\end{tikzcd} 
\end{equation}
Since $X(B)\to B$ is a trivial Kan fibration each 
one of these diagrams has a solution.  The simplicial 
set $B$ is a filtered colimit of its countable 
subcomplexes and hence $X(B)$ is a filtered 
colimit of subcomplexes $X(B')$ with countably 
many non-degenerate simplices.  Since there 
are only countably many such diagrams of the 
form~\eqref{eq:analog of GJ diagram} and 
countably many such diagonal fillers,
 it follows that there is a   
countable subcomplex $D_1\subseteq B$ 
containing $D_0$ such that each 
of the lifting problems above has a solution 
in $X(D_1)$.  Note that $\sk_2D_1 = \sk_2B$ since 
we have the inclusions $\sk_2B \subseteq \sk_2D_1 
\subseteq \sk_2B$.    

Repeat this construction countably 
many times to obtain a sequence $D_0 \subseteq 
D_1 \subseteq D_2 \subseteq \cdots$ of $B$ 
such that all lifting problems of the above form 
over $D_i$ can be solved over $D_{i+1}$.  Let 
$D = \cup_{i\geq 0}D_i$.  Then $D$ is a countable 
subcomplex belonging to $\calB$ which contains the vertices $0$ and 
$1$ and is such that $\set{0}\to D$ is a categorical equivalence.  
This last statement follows from the fact that 
the map $X(D)\to D$ is a trivial Kan fibration since 
each lifting problem 
\[
\begin{tikzcd} 
\partial\Delta^n\arrow[r] \arrow[d] & X(D) \arrow[d] \\ 
\Delta^n \arrow[r] & D 
\end{tikzcd} 
\]
factors through $X(D_i)$ for some $i$ and hence can 
be solved over $D_{i+1}$.    
%
%
%
\end{proof} 

\begin{proof}[Proof of Lemma~\ref{lem:bergners lemma}] 
Since $S$ is a pre-fibrant simplicial set 
and $f\colon x\to y$ is an equivalence 
in $S$, it follows by Remark~\ref{rem:subcomplex S'} 
that there exists a subcomplex $S'\subseteq S$ 
containing $f$ and such that $\h(S') = J$.

Let $\Delta^1\to S'$ classify the edge $f$.  
Choose an inner anodyne map $S'\to K$ where 
$K$ is an $\infty$-category.  Since 
$\h(K) = J$ is a groupoid it follows that 
$K$ is a Kan complex (Corollary 1.4 of \cite{J1}).  The composite 
map $\Delta^1\to S'\to K$ factors through $J$, 
so that we have a commutative diagram 
\[
\begin{tikzcd} 
\Delta^1 \arrow[r] \arrow[d] & J \arrow[d] \\ 
S' \arrow[r] & K 
\end{tikzcd} 
\]
Factor the map $J\to K$ as $J\to Z\to K$ 
where $J\to Z$ is anodyne and $Z\to K$ is 
a Kan fibration.  Let $G$ be defined by 
the pullback diagram 
\[
\begin{tikzcd} 
G \arrow[r] \arrow[d] & Z \arrow[d] \\ 
S' \arrow[r] & K 
\end{tikzcd} 
\]
so that $G\to S'$ is a Kan fibration 
and $G\to Z$ is inner anodyne by 
Lemma~\ref{lem:pullback of inner anodyne along Kan fibn}.  

The map $J\to Z$ is a homotopy equivalence 
between Kan complexes and hence 
is a categorical equivalence 
(Lemma~\ref{lem:hty equiv between Kan complexes}).  It follows that 
the composite map $\set{0}\to J\to Z$ is a 
categorical equivalence.  Therefore the 
composite map $\set{0}\to \Delta^1\to G$ 
is a categorical equivalence.  The 
desired result then follows from 
Lemma~\ref{lem:GJ lemma}.    
\end{proof} 

\section*{Acknowledgements}

I am indebted to Thomas Nikolaus and David Roberts for their comments  
on a first draft of this paper.

\end{document}